\documentclass[final,leqno]{siamltex}
\usepackage{amssymb}
\usepackage{amsmath}

\usepackage{graphicx}
\usepackage{multirow}
\usepackage{subfigure}

\newcommand{\tU}{\widetilde{U}}
\newcommand{\tV}{\widetilde{V}}
\newcommand{\tW}{\widetilde{W}}
\newcommand{\bU}{\mathbf{U}}
\newcommand{\bV}{\mathbf{V}}

\newcommand{\bG}{\mathbf{G}}
\newcommand{\tbU}{\widetilde{\bU}}
\newcommand{\tbV}{\widetilde{\bV}}

\newcommand{\p}{\partial}
\newcommand{\cS}{\mathcal{S}}
\newcommand{\cT}{\mathcal{T}}
\newcommand{\cM}{\mathcal{M}}
\newcommand{\bfR}{\mathbf{R}}
\newcommand{\del}{\delta}
\newcommand{\Ome}{\Omega}
\newcommand{\oOme}{\overline{\Omega}}
\newcommand{\delx}{\delta_x}

\newcommand{\hF}{\widehat{F}}
\newcommand{\hG}{\widehat{G}}
\newcommand{\ou}{\overline{u}}
\newcommand{\uu}{\underline{u}}
\newcommand{\ext}{\mbox{ext}}
\newcommand{\extr}{\mbox{extr}}

\newtheorem{remark}{Remark}
\begin{document}

\title{Convergent finite difference methods for one-dimensional fully 
nonlinear second order partial differential equations}
 
\author{
Xiaobing Feng\thanks{Department of Mathematics, 
The University of Tennessee, Knoxville, TN 37996.
{\tt xfeng@math.utk.edu}. The work of this author 
was partially supported by the NSF grant DMS-0710831.}
\and
%Chiu-Yen Kao\thanks{Department of Mathematics, The Ohio State University, 
%Columbus, OH 43210, U.S.A.  {\tt kao@math.ohio-state.edu},}
Chiu-Yen Kao\thanks{Department of Mathematical Sciences, Claremont Mckenna 
College, Claremont, CA 91711. {\tt Ckao@claremontmckenna.edu} }
\and 
Thomas Lewis\thanks{Department of Mathematics, The University of Tennessee, 
Knoxville, TN 37996. {\tt tlewis@math.utk.edu}. The work of this author 
was partially supported by the NSF grant DMS-0710831.}
}

\maketitle

\begin{abstract}
This paper develops a new framework for designing and analyzing convergent
finite difference methods for approximating both classical and viscosity solutions 
of second order fully nonlinear partial differential equations (PDEs) in 1-D. 
The goal of the paper is to extend the successful 
framework of monotone, consistent, and stable finite difference methods
for first order fully nonlinear Hamilton-Jacobi equations to second order 
fully nonlinear PDEs such as Monge-Amp\`ere and Bellman type equations.
New concepts of consistency, generalized monotonicity, and stability are 
introduced; among them, the generalized monotonicity and consistency, which are 
easier to verify in practice, are natural extensions of the corresponding 
notions of finite difference methods for first order 
fully nonlinear Hamilton-Jacobi equations. The main component of the proposed 
framework is the concept of a ``numerical operator", and the main idea
used to design consistent, generalized monotone and stable finite difference
methods is the concept of a ``numerical moment". These two new concepts play 
the same roles the ``numerical Hamiltonian" and the ``numerical viscosity" 
play in the finite difference framework for first order fully nonlinear 
Hamilton-Jacobi equations.
In the paper, two classes of consistent and monotone finite difference
methods are proposed for second order fully nonlinear PDEs. The first class 
contains Lax-Friedrichs-like methods which also are proved to be stable,
and the second class contains Godunov-like methods. Numerical results 
are also presented to gauge the performance of the proposed finite 
difference methods and to validate the theoretical results of the paper. 
\end{abstract}



\pagestyle{myheadings}
\thispagestyle{plain}
\markboth{X.FENG, C.-Y.KAO, AND T. LEWIS}{FINITE DIFFERENCE METHODS FOR 2ND 
ORDER FULLY NONLINEAR PDEs}

\section{Introduction} \label{sec-1}
Fully nonlinear partial differential equations (PDEs) refers to a class 
of nonlinear PDEs which are nonlinear in the highest order derivatives of 
the unknown functions appearing in the equations. For example, the general 
first and second order fully nonlinear PDEs, respectively, have the form 
$H(\nabla u, u, x)=0$ and $F(D^2 u, \nabla u, u, x)=0$, where 
$\nabla u$ and $D^2 u$ denote the gradient vector and Hessian matrix 
of the unknown function $u$.  Fully nonlinear PDEs, which have experienced 
extensive analytical developments in the past thirty years 
(cf. \cite{Caffarelli_Cabre95, Bardi_Capuzzo97, Gilbarg_Trudinger01, Lieberman96}), 
arise from many scientific and engineering applications such as
differential geometry, astrophysics, antenna design, image processing, optimal
control, optimal mass transport, and geostrophical fluid dynamics. 
Fully nonlinear PDEs play a critical role for the solutions of these 
applications because they appear one way or another in the governing 
equations of these problems.

As expected, the study of first order fully nonlinear PDEs came first. 
Since the introduction of the notion of viscosity solutions by Crandall and 
Lions \cite{Crandall_Lions83} in 1983, the past thirty years has been a period
of explosive developments in analyzing first order 
fully nonlinear PDEs. Starting with the 
pioneering work of Crandall and Lions \cite{Crandall_Lions84},
extensive research has also been successfully carried out on developing numerical
methods, in particular monotone as well as other types of finite difference methods, for 
computing viscosity solutions of first order fully nonlinear PDEs,
especially those arising from the level set formulations of
moving interfaces and those arising from optimal control (cf. \cite{Shu07} and
the references therein). To overcome the low order accuracy barrier of
monotone finite difference methods, various high order local discontinuous
Galerkin (LDG) methods have also been developed recently in the literature
(cf. \cite{Shu07,Yan_Osher11} and the references therein).

In contrast with the success of PDE analysis and numerical approximation for
first order fully nonlinear PDEs, the situation for second order fully 
nonlinear PDEs is very different. On one hand, like in the case of
first order fully nonlinear PDEs, tremendous progresses in PDE analysis have 
been made in the past thirty years (cf. \cite{Gilbarg_Trudinger01,Caffarelli_Cabre95}).
On the other hand, not much progress on developing accurate and efficient 
numerical methods, especially Galerkin-type methods, for second order fully 
nonlinear PDEs has been made until very recently (cf. \cite{Feng_Glowinski_Neilan10,
Feng_Neilan11} and the references therein). The lack of progress is mainly 
due to the following two facts: (i) the 
notion of viscosity solutions is nonvariational; (ii) the conditional
uniqueness (i.e., uniqueness only holds in a restrictive function class)
of viscosity solutions is difficult to handle at the discrete level. 
The first difficulty prevents a direct construction of Galerkin-type methods  
and forces one to use indirect approaches as done in \cite{Dean_Glowinski03, 
Dean_Glowinski06b, Feng_Neilan09a, Feng_Neilan11}
for approximating viscosity solutions. The second difficulty prevents
any straightforward construction of finite difference methods because 
such a method does not have a mechanism to enforce the conditional
uniqueness and often fails to capture the sought-after viscosity 
solution. Since the scope of this paper is confined to the finite difference
method, Galerkin-type methods will not be discussed here. We refer the reader 
to the review paper \cite{Feng_Glowinski_Neilan10} for a detailed discussion 
of recent developments on Galerkin-type methods for second order fully 
nonlinear PDEs.  

The primary goal of this paper is to develop a new framework for 
designing and analyzing convergent finite difference methods for second 
order fully nonlinear (elliptic) PDEs. For the ease of presenting the ideas  
and to observe the page limitation of the journal, we shall only consider 
one-dimensional PDEs in this paper and leave the high dimensional 
generalizations to a forthcoming companion paper \cite{Feng_Kao_Lewis11}.
We use the phrase ``new framework" to distinguish the framework of this 
paper from the existing (abstract) framework originally developed by Barles and
Souganidis in \cite{Barles_Souganidis91} twenty years ago and further
developed recently by Caffarelli and Souganidis in \cite{Caffarelli_Souganidis08}.
Unlike Barles and Souganidis' framework which is abstract and broader
in applications,  our framework is specifically and only designed for finite 
difference methods which can be easily implemented on computers. As a result, 
the proposed framework has the advantages of being simple to understand and easy to 
utilize in practice. Moreover, the new framework is a natural extension
of the successful monotone finite difference framework developed for first order
fully nonlinear Hamilton-Jacobi equations (cf. \cite{Crandall_Lions84,
Shu07} and the references therein).
The main concept of the new framework is the ``numerical operator".  
The key components of the framework are new and easy-to-check notions of 
consistency and generalized monotonicity (g-monotonicity), which 
together with the well-known notion
of stability, form the backbones of the proposed finite difference framework.    
After the framework is established,
one must address a harder question of how to construct specific finite difference 
methods which fulfill the structure conditions (i.e., consistency, 
g-monotonicity, and stability) of the framework in order to make the framework practically useful. 
We note that this question was not addressed in
\cite{Barles_Souganidis91} as the goal of that paper was not to develop
practical numerical methods, and it took seventeen years to construct the first 
finite difference method which fulfills the structure conditions laid out 
in \cite{Barles_Souganidis91} for the second order fully 
nonlinear Monge-Amp\`ere equation in \cite{Oberman08}. 
Moreover, the method
of \cite{Oberman08} is a nonstandard finite difference
method because it requires the use of wide-stencil grids.
We do want to remark that many numerical methods, which 
may or may not fulfill the structure conditions of \cite{Barles_Souganidis91},
have been developed for Bellman type equations (cf. \cite{Barles_Jakobsen07,
Krylov12, Feng_Glowinski_Neilan10} and the references therein).
To address the above key question, our main idea is to introduce a new 
concept called the ``numerical moment". We like to stress that the numerical 
moment not only helps the construction of desired g-monotone finite difference 
methods, but also, we believe, provides a fundamental and indispensable mechanism 
for a finite difference method to overcome the two major difficulties
associated with numerical approximations of second order fully 
nonlinear PDEs. We also note that the new concepts of ``numerical operators"
and ``numerical moments" for second order fully nonlinear PDEs are natural 
extensions of the well-known concepts of ``numerical Hamiltonians" 
and ``numerical viscosities" for first order fully nonlinear
Hamilton-Jacobi equations.

This paper is organized as follows. In Section \ref{sec-2} 
we collect some preliminary materials such as notation and definitions.
In Section \ref{sec-3} we present our finite difference framework. 
The motivation and main ideas are heuristically explained. The 
main concepts and definitions of {\em numerical operators, consistency, 
g-monotonicity, and stability} are formally introduced and defined.
The main result of this section is a convergence theorem
which asserts that the solution of any consistent, g-monotone and stable 
finite difference method is guaranteed to converge to the unique 
viscosity solution of the underlying second order fully nonlinear PDE.
In Section \ref{sec-4} we introduce the concept of a {\em numerical moment}.  
With the help of the numerical moment and the inspiration given by the 
convergent finite difference schemes for first order fully nonlinear Hamilton-Jacobi 
equations, we are able to construct two classes of consistent and g-monotone
finite difference methods. The first class contains Lax-Friedrichs-like 
methods and the second class contains Godunov-like methods. 
By using a non-standard fixed point argument we also prove that every 
consistent and g-monotone Lax-Friedrichs-like method is uniquely solvable 
and stable 
for a given class of fully nonlinear operators.  
In Section \ref{sec-5} we present some detailed numerical results 
to gauge the performance of the proposed finite difference methods and 
to validate the theoretical results of the paper. The paper is
concluded by a short summary in Section \ref{sec-6}.

%%%%%%%%%%%%%%%%%%%%%%%%%%%%%%%%%%%%
\section{Preliminaries}\label{sec-2}

In this paper we adopt standard function and space notations as in 
\cite{Gilbarg_Trudinger01, Caffarelli_Cabre95}. For example, for a bounded 
open domain $\Ome\subset\mathbf{R}^d$, $B(\Ome)$, $USC(\Ome)$ and 
$LSC(\Ome)$ are used to denote, respectively, the spaces of bounded, 
upper semi-continuous and lower semicontinuous functions on $\Ome$.
Also, for any $v\in B(\Ome)$, we define 
\[
v^*(x):=\limsup_{y\to x} v(y) \qquad\mbox{and}\qquad
v_*(x):=\liminf_{y\to x} v(y). 
\]
Then, $v^*\in USC(\Ome)$ and $v_*\in LSC(\Ome)$, and they are called
{\em the upper and lower semicontinuous envelopes} of $v$, respectively.

Given a bounded function $F: \cS^{d\times d}\times\mathbf{R}^d\times 
\mathbf{R}\times \oOme \to \mathbf{R}$, where $\cS^{d\times d}$ denotes the set 
of $d\times d$ symmetric real matrices,  the general second order 
fully nonlinear PDE takes the form
\begin{align}\label{e2.1}
F(D^2u,\nabla u, u, x) = 0 \qquad\mbox{in } \oOme.
\end{align}
Note that here we have used the convention of writing the boundary condition as a
discontinuity of the PDE (cf. \cite[p.274]{Barles_Souganidis91}).

The following two definitions are standard (cf. \cite{Gilbarg_Trudinger01,
Caffarelli_Cabre95,Barles_Souganidis91}).

\begin{definition}\label{def2.1}
Equation \eqref{e2.1} is said to be elliptic if for all 
$(\mathbf{p},\lambda,x)\in \mathbf{R}^d\times \mathbf{R}\times \oOme$ there holds
\begin{align}\label{e2.2}
F(A, \mathbf{p}, \lambda, x) \leq F(B, \mathbf{p}, \lambda, x) \qquad\forall 
A,B\in \cS^{d\times d},\, A\geq B, 
\end{align}
where $A\geq B$ means that $A-B$ is a nonnegative definite matrix.
\end{definition}
We note that when $F$ is differentiable, the ellipticity
also can be defined by requiring that the matrix $\frac{\partial F}{\partial A}$ 
is negative semi-definite (cf. \cite[p. 441]{Gilbarg_Trudinger01}).

\begin{definition}\label{def2.2}
A function $u\in B(\Ome)$ is called a viscosity subsolution (resp. 
supersolution) of \eqref{e2.1} if, for all $\varphi\in C^2(\oOme)$,  
if $u^*-\varphi$ (resp. $u_*-\varphi$) has a local maximum 
(resp. minimum) at $x_0\in \oOme$, then we have
\[
F_*(D^2\varphi(x_0),\nabla \varphi(x_0), u^*(x_0), x_0) \leq 0 
\]
(resp. $F^*(D^2\varphi(x_0),\nabla \varphi(x_0), u_*(x_0), x_0) \geq 0$).
The function $u$ is said to be a viscosity solution of \eqref{e2.1}
if it is simultaneously a viscosity subsolution and a viscosity
supersolution of \eqref{e2.1}.
\end{definition}

%\begin{remark}\label{rem2.1}
%It can be proved that it is sufficient only to consider $\varphi \in\mathbb{P}_2$,
%the space of all {\em quadratic polynomial}, in Definition \ref{def2.2} (see
%\cite[page 20]{Caffarelli_Cabre95}).
%\end{remark}

We remark that if $F$ and $u$ are continuous, then the upper and lower $*$ 
indices can be removed in Definition \ref{def2.2}. The definition
of the ellipticity implies that the differential operator $F$ 
must be non-increasing in its first argument in order to be 
elliptic. It turns out that the ellipticity provides a sufficient 
condition for equation \eqref{e2.1} to fulfill a maximum principle
(cf. \cite{Gilbarg_Trudinger01, Caffarelli_Cabre95}).
It is clear from the above definition that viscosity solutions 
in general do not satisfy the underlying PDEs in a tangible sense, and
the concept of viscosity solutions is {\em nonvariational}. Such
a solution is not defined through integration by parts against arbitrary test 
functions; hence, it does not satisfy an integral identity. As pointed
out in Section \ref{sec-1}, the nonvariational nature of viscosity
solutions is the main obstacle that prevents direct construction 
of Galerkin-type methods, which are based on variational formulations.

%%%%%%%%%%%%%%%%%%%%
\section{A monotone finite difference framework}\label{sec-3}

%To better present the ideas, we only consider the one-dimensional case
%in this subsection. Specifically, 
We consider the following fully nonlinear second order two-point 
boundary value problem:
\begin{alignat}{2}\label{e3.1a}
F(u_{xx}, x) &=0,  &&\quad\qquad a< x<b,\\
               u(a) &=u_a, && \label{e3.1b}\\
               u(b) &=u_b, && \label{e3.1c}
\end{alignat}
where $u_a$ and $u_b$ are two given numbers and  $F$ is assumed to be an elliptic
operator in a function class ${\cal A}\subset C^0(\Omega)$.
We remark that the results of this paper can be easily extended to PDEs 
with general form $F(u_{xx}, u_x, u, x) =0$. 

To construct finite difference methods for the above problem, we first need 
to have a mesh for the domain/interval $\Omega:=(a,b)$. For simplicity, 
we only consider uniform meshes here, although our methods can be easily 
generalized to 
nonuniform meshes.  Let $J$ be a positive integer and $h=\frac{b-a}{J-1}$.  
We divide $\Omega$ into $J-1$ subintervals/subdomains with grid points
$x_j=a+ (j-1)h$ for $j=1,2,\cdots,J$, and let $\cT_h=\{x_j\}_{j=1}^J$ be a mesh
of $\oOme$. Define the forward and backward difference operators by
\[
\del_x^+ v(x):= \frac{v(x+h) - v(x)}{h},\qquad
\del_x^- v(x):= \frac{v(x)- v(x-h)}{h},
\]
for a continuous function $v$ defined in $\Ome$ and 
\[
\del_x^+ V_j:= \frac{V_{j+1} - V_j}{h},\qquad
\del_x^- V_j:= \frac{V_j- V_{j-1}}{h},
\]
for a grid function $V$ defined on the mesh $\cT_h$. The operators 
$\delx^+$ and $\delx^-$ will serve as building blocks in the construction of 
our finite difference methods in the sense that we approximate
all first and second derivatives by using combinations and
compositions of these two operators.

To approximate $u_x(x_j)$, we have two options
\[
u_x(x_j)\approx \delx^+ u(x_j), \qquad\qquad
u_x(x_j)\approx \delx^- u(x_j).
\]
As a result, we have three possible ways to approximate $u_{xx}(x_j)$ given by
\begin{align*}
u_{xx}(x_j) &\approx \delx^+ \delx^+ u(x_j),\hskip 1.2in
u_{xx}(x_j) \approx \delx^- \delx^- u(x_j),\\
u_{xx}(x_j) &\approx \delx^+ \delx^- u(x_j) = \delx^- \delx^+  u(x_j).
\end{align*}
It is easy to verify that
\[
\delx^+ \delx^+ u(x_j) = \delx^2 u(x_{j+1}), \quad
\delx^- \delx^- u(x_j) = \delx^2 u(x_{j-1}), \quad
\delx^- \delx^+ u(x_j) = \delx^2 u(x_j),
\]
where 
\[
\delx^2 v(x):=\frac{v(x-h) -2v(x)+v(x+h)}{h^2}
\]
for a continuous function $v$ and 
\[
\delx^2 V_j:=\frac{V_{j+1} -2V_j+V_{j-1} }{h^2}
\]
for a grid function $V$ on the mesh $\cT_h$.

The above simple argument motivates us to propose the following
general finite difference method for equation \eqref{e3.1a}:
Find a grid function $U$ such that
\begin{align}\label{fdm}
\hF(\delx^2 U_{j-1}, \delx^2 U_j, \delx^2 U_{j+1}, x_j)=0 
\end{align}
for $j=2,3,\cdots,J-1$. As expected, $U_j$ is intended to be an approximation 
of $u(x_j)$ for $j=1,2,\cdots, J$, and $U_0$ and $U_{J+1}$ are two ghost values. 

\begin{definition}\label{def1}
The function $\hF: \bfR\times\bfR\times\bfR\times\bfR\to \bfR$ in \eqref{fdm} 
is called  a {\em numerical operator}. Finite difference method \eqref{fdm} is
said to be an {\em admissible} scheme for problem \eqref{e3.1a}--\eqref{e3.1c}
if it has at least one (grid function) solution $U$  
such that $U_1=u_a$ and $U_J=u_b$.
\end{definition}

It is easy to understand that $\hF$ needs to be some approximation of the 
differential operator $F$ in order for scheme \eqref{fdm} to be relevant to
the original PDE problem.  Generally, different {\em numerical operators} 
$\hF$ should result in different finite difference methods. A natural and 
important question is how to construct $\hF$. We shall defer answering 
this question to the next section where we present two types 
of {\em numerical operators} $\hF$. For now, we 
propose a set of conditions (or properties) which we like 
to impose on $\hF$. We choose conditions such that if $\hF$ satisfies
them, then the solution of the finite difference method \eqref{fdm}
is guaranteed to converge to the viscosity solution of problem 
\eqref{e3.1a}--\eqref{e3.1c}. The conditions will be reflected
in the following definition.

\begin{definition}\label{def3.1}
\begin{itemize}
\item[{\rm (i)}] Finite difference method \eqref{fdm} is said to be 
a {\em consistent} scheme if $\hF$ satisfies  
\begin{align}\label{A1a}
\liminf_{p_k\to p, k=1,2,3\atop \xi\to x} \hF(p_1,p_2,p_3, \xi) 
&\geq F_*(p, x),\\
\limsup_{p_k\to p, k=1,2,3 \atop \xi\to x} \hF(p_1,p_2,p_3, \xi) 
&\leq F^*(p, x), \label{A1b} \\
\liminf_{p_k\to -\infty, k=1,2,3\atop \xi\to x} \hF(p_1,p_2,p_3, \xi) 
&\geq F_*(-\infty, x) := \liminf_{p \to -\infty} F(p,x), \label{A1c}\\
\limsup_{p_k\to \infty, k=1,2,3 \atop \xi\to x} \hF(p_1,p_2,p_3, \xi) 
&\leq F^*(\infty, x) := \limsup_{p \to \infty} F(p,x), \label{A1d} 
\end{align}
for $p\in \bfR$, where $F_*$ and $F^*$ denote respectively the lower and the upper
semi-continuous envelopes of $F$. 

\item[{\rm (ii)}] Finite difference method \eqref{fdm} is said to be 
a {\em g-monotone} scheme if for each $2\leq j\leq J-1$,  $\hF(p_1,p_2,p_3,x_j)$ 
is monotone increasing in $p_1$ and $p_3$ and monotone decreasing in $p_2$;
that is, $\hF(\uparrow,\downarrow,\uparrow, x_j)$ for $j=2,3,\cdots,J-1$.

\item[{\rm (iii)}] Let \eqref{fdm} be an admissible finite difference 
method.  A solution $U$ of \eqref{fdm} is said to be 
{\em stable} if there exists a constant $C>0$, which is independent 
of $h$, such that $U$ satisfies
\begin{equation}\label{A3}
\|U\|_{\ell^\infty(\cT_h)}:=\max_{1\leq j\leq J} |U_j| \leq C. 
\end{equation}
Also, \eqref{fdm} is said to be a {\em stable scheme} if all of its solutions 
are stable solutions.

\end{itemize}

\end{definition}

\begin{remark}\label{rem3.1}
(a) The consistency and g-monotonicity (generalized monotonicity)
defined above are different from those given in 
\cite{Barles_Souganidis91,Kuo_Trudinger92,
Caffarelli_Souganidis08}. $\hF$ is asked to be monotone in $\delta_x^2 U_{j-1},
\delta_x^2 U_j$ and $\delta_x^2 U_{j+1}$, not in each individual 
entry $U_j$. To avoid confusion, we use the words ``g-monotonicity" and
``g-monotone" to indicate that the monotonicity is defined as above.  
We shall demonstrate in the next section that the above new 
definitions, especially the one for g-monotonicity, are more suitable 
and much easier to verify for (practical) finite difference methods. 
The new notions of consistency and g-monotonicity are logical 
extensions of their widely used counterparts for the first order 
Hamilton-Jacobi equations \cite{Crandall_Lions84,Shu07}.

(b) On the other hand, the above stability definition is the same as
that given in \cite{Barles_Souganidis91,Kuo_Trudinger92,
Caffarelli_Souganidis08}. 
%For a specific finite difference method, verifying the validity of the stability, 
%not the g-monotonicity, is often a more difficult task to 
%accomplish in our finite difference framework.
%We shall prove in Section \ref{sec-4} that the g-monotonicity
%actually implies the stability and admissibility, see Theorem \ref{thm4.2}.

(c) We note that if $F$ is a continuous function, we can also assume
that $\hF$ is a continuous function. Then, \eqref{A1a} and \eqref{A1b} 
reduce to the condition $\hF(p,p,p, x)=F(p,x)$.

%(d) Due to the fact we are approximating viscosity solutions, we may have
%$p=\pm \infty$ in \eqref{A1a}, \eqref{A1b}.  Note that the definitions of $F^*$ and $F_*$
%are well-defined for $F : \overline{\bfR} \times \Omega \to \overline{\bfR}$.
%As we will see in Theorem \ref{con_thm}, stability for difference quotients 
%is not necessary. 

(d) The ``good"  numerical operators $\hF$ we construct so far 
(cf. Section \ref{sec-4}) all have the form 
\begin{equation}\label{hG}
\hF(p_1, p_2, p_3, \xi)= \hG \big(\overline{p_2}, p_2, \xi \big)
\end{equation}
for some function $\hG$ and $\overline{p_2} := (p_1 + p_3)/2$. 
In other words,  $\hF$ is a function of $\overline{p_2}$ and $p_2$. 
Hence,  a g-monotone $\hF$ should be  increasing in $p_1+p_3$ and 
decreasing in $p_2$.  In this case, the consistency condition reduces to 
\begin{align}\label{hG_1}
\liminf_{\sigma_1,\sigma_2\to p\atop \xi\to x} \hG(\sigma_1,\sigma_2, \xi) 
&\geq F_*(p, x),\\
\limsup_{\sigma_1,\sigma_2\to p\atop \xi\to x} \hG(\sigma_1,\sigma_2, \xi) 
&\leq F^*(p, x), \label{hG_2} \\
\liminf_{\sigma_1,\sigma_2\to -\infty\atop \xi\to x} \hG(\sigma_1,\sigma_2, \xi) 
&\geq F_*(-\infty, x) := \liminf_{p \to -\infty} F(p,x), \label{hG_3}\\
\limsup_{\sigma_1,\sigma_2\to \infty\atop \xi\to x} \hG(\sigma_1,\sigma_2, \xi) 
&\leq F^*(\infty, x) := \limsup_{p \to \infty} F(p,x). \label{hG_4}
\end{align}
We shall need to use the above form of $\hF$ in the proof of our 
convergence theorem, see Theorem \ref{con_thm} below.
\end{remark}

For a given grid function $U$, we define a piecewise constant 
extension function $u_h$ of $U$ as follows:
\begin{align}\label{def_uh} 
u_h(x) := U_j \qquad\forall x\in (x_{j-\frac12}, x_{j+\frac12}], \quad
j=1,2,\cdots, J,
\end{align}
where $x_{j\pm \frac12}= x_j\pm \frac{h}2$ for $j=1,2,\cdots, J$.

\begin{definition}\label{def3.2}
Problem \eqref{e3.1a}--\eqref{e3.1c} is said to satisfy a {\em comparison 
principle} if the following statement holds. For any upper semi-continuous 
function $u$ and lower semi-continuous function $v$ on $\overline{\Omega}$,  
if $u$ is a viscosity subsolution and $v$ is a viscosity supersolution 
of \eqref{e3.1a}--\eqref{e3.1c}, then $u\leq v$ on $\overline{\Omega}$.
\end{definition}

\begin{remark}
Since the comparison principle immediately infers the uniqueness of viscosity 
solutions, it is also called a {\em strong uniqueness property} 
for problem \eqref{e3.1a}--\eqref{e3.1c} (cf. \cite{Barles_Souganidis91}).
\end{remark}

We are now ready to state and prove the following convergence theorem, 
which is the main result of this paper.

\begin{theorem}\label{con_thm}
Suppose problem \eqref{e3.1a}--\eqref{e3.1c} satisfies the comparison 
principle of Definition \ref{def3.2} and has a unique continuous viscosity 
solution $u$. Let $U$ be a solution to a consistent, g-monotone, and stable 
finite difference method \eqref{fdm} with $\hF$ satisfying \eqref{hG}, and let $u_h$ be its 
piecewise constant extension as defined above.
%Suppose that $\hF$ is a continuous function in all its arguments. 
Then $u_h$ converges to $u$ locally uniformly as $h\to  0^+$. 
\end{theorem}

\begin{proof}
We divide the proof into five steps. 

{\em Step 1}: Since $U$ satisfies \eqref{A3}, it is trivial to check 
that $u_h$ satisfies
\begin{equation}\label{e3.7}
\|u_h\|_{L^\infty(\Ome)} \leq C.
\end{equation}
Define $\ou, \uu\in L^\infty(\Ome)$ by 
\[
\ou(x):=\limsup_{\xi\to x\atop h\to 0^+} u_h(\xi), 
\qquad \uu(x):=\liminf_{\xi\to x\atop h\to 0^+} u_h(\xi).
\]
We now show that $\ou$ and $\uu$ are, respectively, a viscosity 
subsolution and a viscosity supersolution of \eqref{e3.1a}--\eqref{e3.1c}.
Hence, they must coincide by the comparison principle.

Suppose that $\ou-\varphi$ takes a local maximum 
at $x_0\in \Ome$ for some $\varphi\in C^2(\oOme)$.
We first assume that $\varphi\in \mathbb{P}_2$, the set
of all quadratic polynomials. In {\em Step 3} we will
consider the general case $\varphi\in C^2(\oOme)$. 
Without loss of generality, we assume $\ou(x_0)-\varphi(x_0)$ is a 
strict local maximum and $\ou(x_0)=\varphi(x_0)$ (after a translation 
in the dependent variable).  Then there exists a ball/interval, 
$B_{r_0}(x_0)$, centered at $x_0$ with radius $r_0>0$ such that
\begin{equation}\label{e3.9}
\ou(x)-\varphi(x) < \ou(x_0)-\varphi(x_0)=0 \qquad\forall x\in B_{r_0}(x_0).
\end{equation}
Thus, there exists sequences $\{h_k\}_{k\geq 1}$ and $\{\xi_k\}_{k\geq 1}$ 
such that as $k\to \infty$,
\begin{align*}
&h_k\to 0^+,\qquad \xi_k\to x_0, \qquad u_{h_k}(\xi_k)\to \ou(x_0),  \\
& u_{h_k}(x)-\varphi(x) \mbox{ takes a local maximum at $\xi_k$ for sufficiently 
large $k$},  
%\\
%& \lim_{k \to \infty} \delta_{x, h_k}^2  \bigl( u_{h_k} (\xi_k) - \varphi (\xi_k) \bigr) 
%= \liminf_{h \to 0} \delta_{x, h}^2  \bigl( \ou(x_0) - \varphi (x_0) \bigr), 
\end{align*}
and 
\begin{equation}\label{e3.10b}
\lim_{k \to \infty} \delta_{x, h_k}^2 u_{h_k} (\xi_k) = \liminf_{h \to 0} \delta_{x,h}^2 \ou(x_0), 
\end{equation}
where
\[
\delta_{x,\rho}^2 u_h(\xi):=\frac{ u_h(\xi-\rho) -2u_h(\xi) + u_h(\xi+\rho) }{\rho^2}
\quad \forall \xi\in (a+\rho,b-\rho), \,\, \rho>0. 
\] 
We remark that the right-hand side of \eqref{e3.10b} could either be finite or 
negative infinite.

Then, there exists $k_0>>1$ such that $h_k< r_0$ and 
\begin{equation} \label{e3.10}
0\stackrel{k\to \infty}{\longleftarrow} u_{h_k}(\xi_k)-\varphi(\xi_k) 
\geq u_{h_k}(x)-\varphi(x) \quad\forall x\in B_{r_0}(x_0), \,\, k\geq k_0.
\end{equation}
%Using the fact that $\varphi$ is a quadratic function, we have 
%\begin{equation}\label{e3.10b}
%\lim_{k \to \infty} \delta_{x, h_k}^2 u_{h_k} (\xi_k) = \liminf_{h \to 0} \delta_{x,h}^2 \ou(x_0). 
%\end{equation}
%We remark that the right-hand side of \eqref{e3.10b} could be finite and could be negative infinite. 

\medskip
{\em Step 2}:  Since $U$ satisfies \eqref{fdm} with $\hF$ being of the form \eqref{hG} at every interior grid point,  
it is easy to check that for $x\in \Ome_h:=( a+\frac{3h}2, b-\frac{3h}2)$, 
\begin{align}\label{e3.11}
0&=\hF\bigl(\delta_{x,h}^2 u_{h}(x-h), \delta_{x,h}^2 u_{h}(x),\delta_{x,h}^2 u_h(x+h), x) \\
&=\hG\bigl(\delta_{x,h}^{\overline{2}} u_h(x), \delta_{x,h}^2 u_h(x),x\bigr), \nonumber
\end{align}
where
\[
\delta_{x,h}^{\overline{2}} u_h(x):= \delta_{x,h}^2 u_h(x-h)+ \delta_{x,h}^2 u_h(x+h). 
\]

Since $u_{h_k}(x)-\varphi(x)$ takes a local maximum at $\xi_k$ 
and $h_k< r_0$ for $k\geq k_0$, by \eqref{e3.10} we have 
\begin{equation}\label{e3.14}
\delta_{x,h_k}^2 u_{h_k}(\xi_k) \leq \delta_{x,h_k}^2 \varphi(\xi_k) 
=\varphi_{xx} (x_0)
\qquad \forall k\geq k_0.
\end{equation}
Also, by \eqref{e3.9}, we get 
\[
\delta_{x,h}^2 \ou(x_0) \leq \delta_{x,h}^2 \varphi(x_0) =\varphi_{xx} (x_0)
\qquad \forall h\leq  r_0.
\]
Thus,
\begin{equation}\label{e3.14a}
\limsup_{h\to 0} \delta_{x,h}^2 \ou(x_0) \leq   \varphi_{xx} (x_0).
\end{equation}

Next, a direct computation yields that
\begin{equation}\label{e3.12}
\delta_{x,h}^{\overline{2}} u_h(x) = \delta_{x,h}^2 u_h(x) + 2R_h u_h(x), 
\end{equation}
where
\[
R_h u_h(x):= \delta_{x,2h}^2 u_h(x) - \delta_{x,h}^2 u_h(x).
\]
By \eqref{e3.10b} and the definition of $\liminf$ we get 
\begin{align}\label{doubted}
	\liminf_{k \to \infty} \delta_{x,2h_k}^2 u_{h_k}(\xi_k)
	 &= \liminf_{k \to \infty} \delta_{x, 2h_k}^2 \ou (x_0) \\
	&\geq \liminf_{h \to 0} \delta_{x,h}^2 \ou (x_0)
	= \lim_{k \to \infty} \delta_{x, h_k}^2 u_{h_k}(\xi_k). \nonumber
\end{align}
Thus, 
\begin{align} \label{rh}
	\liminf_{k \to \infty} R_{h_k} u_{h_k}(\xi_k) 
	%= \liminf_{k \to \infty} \bigl( \delta_{x,2h_k}^2 u_{h_k}(\xi_k) - \delta_{x,h_k}^2   u_{h_k}(\xi_k) \bigr) \geq 0,
	%\nonumber
	 = \liminf_{k \to \infty} \delta_{x,2h_k}^2 u_{h_k}(\xi_k) 
	   - \lim_{k \to \infty} \delta_{x,h_k}^2 u_{h_k}(\xi_k)  \geq 0,
	%\nonumber
	%& \geq 0, 
\end{align}
and there exists a sequence $\{ \epsilon_k \}_{k \geq 1}$ and a constant $k_1 >> 1$ such that 
\begin{subequations} \label{db}
\begin{align}\label{db_a}
	 \delta_{x,h_k}^{\overline{2}} u_{h_k}(\xi_k) &\geq \delta_{x, h_k}^2  u_{h_k}(\xi_k) + \epsilon_k, \qquad \forall k \geq k_1, \\
	 \lim_{k \to \infty} \epsilon_k &= 0 \label{db_b}
\end{align}
\end{subequations}
by \eqref{e3.12} and \eqref{rh}.
  
Now, it follows from \eqref{e3.11}, \eqref{db_a},  and the 
g-monotonicity of the numerical operator $\hF$ (or $\hG$) that for $k\geq \max\{k_0,k_1\}$,
\begin{align*}
0 &=\hF\bigl(\delta_{x,h_k}^2 u_{h_k}(\xi_k-h_k), \delta_{x,h_k}^2 u_{h_k}(\xi_k),
\delta_{x,h_k}^2 u_{h_k}(\xi_k+h_k),\xi_k\bigr) \\
 &=\hG(\delta_{x,h_k}^{\overline{2}} u_{h_k}(\xi_k), \delta_{x,h_k}^2 u_{h_k}(\xi_k),\xi_k \bigr)  \\
&\geq \hG\bigl(\delta_{x,h_k}^2 u_{h_k}(\xi_k) + \epsilon_k,\delta_{x,h_k}^2 u_{h_k}(\xi_k) ,\xi_k\bigr).
\end{align*}
Thus, by \eqref{e3.10b}, \eqref{db_b},  the consistency of $\hF$ (or $\hG$), and \eqref{e3.14a} we get
\begin{align*}
0 &= \liminf_{k\to \infty} \hF\bigl(\delta_{x,h_k}^2 u_{h_k}(\xi_k-h_k), 
\delta_{x,h_k}^2 u_{h_k}(\xi_k), \delta_{x,h_k}^2 u_{h_k}(\xi_k+h_k),\xi_k\bigr)\\
&= \liminf_{k\to \infty} \hG\bigl(\delta_{x,h_k}^2 u_{h_k}(\xi_k) + \epsilon_k,\delta_{x,h_k}^2 u_{h_k}(\xi_k),\xi_k\bigr) \\
& \geq F_*(\lim_{k\to\infty} \delta_{x,h_k}^2 u_{h_k}(\xi_k), x_0) \nonumber \\
&= F_*(\liminf_{h \to 0} \delta_{x,h}^2 \ou(x_0), x_0) \nonumber\\
&\geq F_*(\limsup_{h \to 0} \delta_{x,h}^2 \ou(x_0), x_0) \nonumber\\
&\geq F_*( \varphi_{xx}(x_0), x_0), \nonumber
\end{align*}
where we have used the fact that $F_*$ is decreasing in its first argument to obtain the last two inequalities. This 
is true by the definition of $F_*$ and Definition \ref{def1}.  

\medskip
{\em Step 3}: We consider the general case $\varphi \in C^2(\oOme)$ which
is alluded in {\em Step 2}. Recall that $\ou-\varphi$ is assumed to have 
a local maximum at $x_0$. Using Taylor's formula we write
\begin{align*}
\varphi(x) &=  \varphi(x_0) + \varphi_x(x_0) (x-x_0) 
+\frac12 \varphi_{xx}(x_0) (x-x_0)^2 + o(|x-x_0|^2) \\
&:=p(x) + o(|x-x_0|^2). \nonumber
\end{align*}
For any $\epsilon>0$, we define the following quadratic polynomial: 
\begin{align*}
p^\epsilon(x) &:= p(x) + \epsilon (x-x_0)^2 \\
&= \varphi(x_0) + \varphi_x(x_0) (x-x_0) 
+ \Bigl[\epsilon+\frac{\varphi_{xx}(x_0)}2\Bigr] (x-x_0)^2 .
\end{align*}

Trivially, $p^\epsilon_{xx}(x) = 2\epsilon + \varphi_{xx}(x_0)$ and
$\varphi(x)-p^\epsilon(x) = o(|x-x_0|^2)- \epsilon (x-x_0)^2 \leq 0$. 
Thus, $\varphi-p^\epsilon$ 
has a local maximum at $x_0$, Therefore, $\ou- p^\epsilon$ has a local
maximum at $x_0$.  By the result of {\em Step 2} we have 
$F_*(p^\epsilon_{xx}(x_0), x_0)\leq 0$, that is,  
$F_*(2\epsilon +\varphi_{xx}(x_0), x_0)\leq 0$.  Taking 
$\liminf_{\epsilon\to 0}$ and using the lower semicontinuity of $F_*$ 
we obtain $0\geq \liminf_{\epsilon\to 0} F_*(2\epsilon +\varphi_{xx}(x_0), x_0)
\geq  F_*(\varphi_{xx}(x_0), x_0)$. Thus, $\ou$ is a viscosity subsolution 
of \eqref{e3.1a}--\eqref{e3.1c}.

\medskip
{\em Step 4}: By following almost the same lines as those of Step 2 and 3, 
we can show that if $\uu-\varphi$ takes a local minimum at $x_0\in \Omega$ 
for some $\varphi\in C^2(\oOme)$, then $F^*(\varphi_{xx}(x_0), x_0)\geq 0$. 
Hence,  $\uu$ is a viscosity supersolution of \eqref{e3.1a}--\eqref{e3.1c}. 

\medskip
{\em Step 5}: By the comparison principle (see Definition \ref{def3.2}), we
get $\ou\leq \uu$ on $\Omega$. On the other hand, by their definitions,
we have $\uu\leq \ou$ on $\Omega$. Thus, $\ou=\uu$, which coincides 
with the unique continuous viscosity solution $u$ of 
\eqref{e3.1a}--\eqref{e3.1c}. The proof is complete.
\end{proof}

%%%%%%%%%%%%%%%%
\section{Two types of g-monotone finite difference methods}\label{sec-4}
In this section we first construct two classes of practical finite difference
methods of the form \eqref{fdm}. Using the first class of methods as examples,
we then go through all the steps for verifying the assumptions of 
Theorem \ref{con_thm}, in particular, to present a fixed point argument
for verifying the admissibility and stability.  

\subsection{Finite difference methods with explicit numerical moments}
\label{sec-4.1}

We propose the following family of schemes with numerical operators:
\begin{equation}\label{example123}
\hF_\beta(p_1,p_2,p_3,x):= F(\beta_1 p_1 + \beta_2 p_2 + \beta_3 p_3, x) 
+\alpha \bigl(p_1 -2p_2 +p_3\bigr),
\end{equation}
where $\{\beta_j\}_{j=1}^3$ are nonnegative constants satisfying 
$\beta_1+\beta_2+\beta_3=1$, and  $\alpha$ is an underdetermined 
positive constant or function.

Some specific examples from this family are  
\begin{align}\label{example1}
\hF_1(p_1,p_2,p_3,x)&:= F\Bigl(\frac{p_1 +p_2 +p_3}{3},x\Bigr) 
+\alpha \bigl(p_1 -2p_2 +p_3\bigr), \\
\hF_2(p_1,p_2,p_3,x)&:= F(p_2,x) +\alpha \bigl(p_1 -2p_2 +p_3\bigr), \label{example2}\\
%\hF_3(p_1,p_2,p_3,x)&:= F\Bigl(\frac{p_1 +p_3}{2}, x\Bigr) 
%+\alpha \bigl(p_1 -2p_2 +p_3\bigr). \label{example3}
\hF_3(p_1,p_2,p_3,x)&:= F\Bigl(\frac{p_1 +2p_2 +p_3}{4},x\Bigr) 
+\alpha \bigl(p_1 -2p_2 +p_3\bigr). \label{example3}
\end{align}

\begin{remark}
The term $\alpha \bigl(p_1 -2p_2 +p_3\bigr)$ is called a {\em numerical moment} 
due to the fact 
\[
	\delx^2 U_{j-1} - 2\delx^2 U_{j} + \delx^2 U_{j+1}
	= h^2 \, \frac{U_{j-2} - 4 U_{j-1} + 6 U_j - 4 U_{j+1} + U_{j+2}}{h^4},
\]
a central difference approximation of $u_{x x x x}(x_j)$ scaled by $h^2$.
\end{remark}

%%%%%%%%%%%
\subsection{Finite difference methods without explicit numerical moments}
\label{sec-4.2}

Given $p_1,p_2, p_3\in \mathbf{R}$, let $I(p_1,p_2, p_3)$ denote the smallest 
interval that contains $p_1,p_2$ and $p_3$, that is,
\[
I(p_1,p_2, p_3):= \bigl[ \min\{p_1,p_2, p_3\}, \max\{p_1,p_2, p_3\} \bigr].
\]
Our first method in this family is the following Godunov type scheme 
(cf. \cite{Shu07} and the references therein). Its numerical 
operator $\hF_4$ is defined by 
\begin{equation}\label{example4}
\hF_4(p_1,p_2,p_3,x):= \underset{p\in I(p_1,p_2, p_3)}{\ext}\, F(p,x), 
\end{equation}
where
\begin{equation}\label{example4a}
\underset{p\in I(p_1,p_2, p_3)}{\ext}
:= \begin{cases}
%\underset{ \min\{p_1,p_3\} \leq p\leq p_2}{\mbox{min}}  
\underset{p\in I(p_1,p_2, p_3)}{\mbox{min}}  
&\qquad\mbox{if } p_2\geq \max\{p_1,p_3\}, \\ 
%\underset{p_2\leq p\leq \max\{p_1,p_3\}}{\mbox{max}} 
\underset{p\in I(p_1,p_2, p_3)}{\mbox{max}}  
&\qquad\mbox{if } p_2\leq \min\{p_1,p_3\},\\
\underset{ p_1 \leq p\leq p_2}{\mbox{min}}  &\qquad\mbox{if } p_1<p_2<p_3, \\ 
\underset{ p_3 \leq p\leq p_2}{\mbox{min}}  &\qquad\mbox{if } p_3<p_2<p_1.
\end{cases} 
\end{equation}

Our second method in this family is a slight modification of the 
previous scheme, and its numerical operator, $\hF_5$, is defined by
\begin{align}\label{example5}
\hF_5(p_1,p_2,p_3,x):= \underset{p\in I(p_1,p_2, p_3)}{\extr}\, F(p,x), 
\end{align}
where
\begin{equation}\label{example5a}
\underset{p\in I(p_1,p_2, p_3)}{\extr}
:= \begin{cases}
\underset{p\in I(p_1,p_2, p_3)}{\mbox{min}}  
&\qquad\mbox{if } p_2\geq \max\{p_1,p_3\}, \\ 
\underset{p\in I(p_1,p_2, p_3)}{\mbox{max}}  
&\qquad\mbox{if } p_2\leq \min\{p_1,p_3\},\\
\underset{ p_2 \leq p\leq p_3}{\mbox{max}}  &\qquad\mbox{if } p_1<p_2<p_3, \\ 
\underset{ p_2 \leq p\leq p_1}{\mbox{max}}  &\qquad\mbox{if } p_3<p_2<p_1.
\end{cases} 
\end{equation}

It is not hard to check that both $\hF_4$ and $\hF_5$ are consistent 
and g-monotone numerical operators.

%Finally, we define two more methods in this family whose numerical 
%operators, $\hF_6$ and $\hF_7$, are defined by
%\begin{align}\label{example6}
%&\hF_6(p_1,p_2,p_3):= \underset{p\in I(p_1,p_2, p_3)}{\mbox{min}}\, F(p),\\
%&\hF_7(p_1,p_2,p_3):= \underset{p\in I(p_1,p_2, p_3)}{\mbox{max}}\, F(p).
%\label{example7}
%\end{align}

%%%%%%%%%%%
\subsection{Verification of consistency, g-monotonicity, admissibility 
and stability for scheme \eqref{example123}} \label{sec-4.3}

In this subsection we use the methods with numerical operator $\hF_\beta$ 
as examples to demonstrate all the steps for verifying the assumptions 
of the convergence theorem, Theorem \ref{con_thm}. As mentioned before, 
the consistency and g-monotonicity are easy to verify,  but 
the verification of the admissibility and stability are more involved.
For simplicity, we only consider the case that $F$ is differentiable
and there exists a positive constant $\gamma > 0$ such that
\begin{equation} \label{LF_C}
	0 > -1/\gamma \geq \frac{\p F}{\p p} \geq -\gamma.
\end{equation}
Recall that 
\[
\hF_\beta(p_1,p_2,p_3,x):= F\bigl(\beta_1 p_1 + \beta_2 p_2 +\beta_3 p_3,x\bigr) 
+\alpha \bigl(p_1 -2p_2 +p_3\bigr), \\
\]
where $\beta_1, \beta_2$ and $\beta_3$ are nonnegative constants such that
$\beta_1+\beta_2+\beta_3=1$.

Trivially, $\hF_\beta(p,p,p,x)=F(p,x)$. Hence, $\hF_\beta$ is a consistent 
numerical operator for each set of $\beta_1, \beta_2$ and $\beta_3$
(see Remark \ref{rem3.1} (c)). To verify the g-monotonicity, we compute
\[
\frac{\p \hF_\beta}{\p p_1}=\beta_1 \frac{\p F}{\p p} + \alpha, \qquad
\frac{\p \hF_\beta}{\p p_2}=\beta_2 \frac{\p F}{\p p} -2 \alpha, \qquad
\frac{\p \hF_\beta}{\p p_3}=\beta_3 \frac{\p F}{\p p} + \alpha.
\]
Then $\hF_\beta$ is g-monotone if
\[
\frac{\p \hF_\beta}{\p p_1}> 0,\qquad \frac{\p \hF_\beta}{\p p_2}< 0,\qquad
\frac{\p \hF_\beta}{\p p_3}> 0.
\]
On noting that $\frac{\p F}{\p p} \leq 0$, solving the above system of 
inequalities yields
\begin{equation}\label{e4,48}
\alpha > -\max\{\beta_1, \beta_3\}\, \frac{\p F}{\p p}.
\end{equation}
Thus, we have proved the following theorem.

\begin{theorem}\label{thm4.1}  
$\hF_\beta$ is g-monotone provided that
\begin{equation}\label{e4.49}
\alpha> \max\{\beta_1, \beta_3\}\,
%\max_{p\in [p_{\mbox{\tiny min}}, p_{\mbox{\tiny max}}]}\left[-\frac{\p F(p)}{\p p}\right]
\gamma 
\end{equation}
%and the right-hand side is bounded. 
for $\gamma$ defined by \eqref{LF_C}.
\end{theorem}
%We remark that in case $F$ is not differentiable, one must verify the g-monotonicity 
%differently.

Next, we verify the admissibility and stability of the schemes. To this end, 
we consider the mapping $\cM_\rho: U\to \tU$ defined by 
\begin{equation}\label{e4.50}
\delta_x^2\tU_j =\delta_x^2 U_j 
+ \rho \hF_\beta\bigl(\delta_x^2 U_{j-1},\delta_x^2 U_j,
\delta_x^2 U_{j+1}, x_j\bigr), \quad j=2,3,\cdots, J-1.
\end{equation}
Let $\bU:=( U_2,U_3,\cdots,U_{J-1} )^T$ and $\tbU:=( \tU_2,\tU_3, \cdots, \tU_{J-1} )^T$.
Then \eqref{e4.50} can be rewritten in vector form as 
\begin{equation}\label{e4.51}
A\tbU =A\bU + \rho \bG(\bU),
\end{equation}
where $A$ stands for the tridiagonal matrix corresponding to the 
difference operator $\delta_x^2 U_j$ and $\bG(\bU)=\bigl(G_2(\bU,x_2),G_3(\bU,x_3), 
\cdots,G_{J-1}(\bU,x_{J-1}) \bigr)^T$ with
\[
G_j(\bU,x_j) = \hF_\beta\bigl(\delta_x^2 U_{j-1},\delta_x^2 U_j,
\delta_x^2 U_{j+1}, x_j\bigr), \quad j=2,3,\cdots, J-1.
\]
$\cM_\rho$ is said to be {\em monotone} if $\tbU$ is increasing in each 
component of $\bU$.

\begin{proposition}\label{prop4.1}
Suppose that $\hF_\beta$ is g-monotone, that is, \eqref{e4.49} holds.  Then the 
mapping $\cM_\rho$ is monotone for sufficiently small $\rho>0$.  
\end{proposition}

\begin{proof}
Consider the following system
\begin{align}\label{e4.52}
W_j &=\delta_x^2 U_j, \quad j=2,3,\cdots, J-1, \\
\tW_j &=W_j +\rho \hF_\beta(W_{j-1}, W_j, W_{j+1}, x_j), 
\quad j=2,3,\cdots, J-1, \label{e4.53}\\
\delta_x^2\tU_j &= \tW_j, \quad j=2,3,\cdots, J-1. \label{e4.54}
\end{align}
Let $\cM^{(1)}: U \to W$, $\cM^{(2)}_\rho: W \to \tW$,
and $\cM^{(3)}: \tW\to \tU$. Then, it is easy to verify that
$\cM_\rho$ can be written as a composition operator of $\cM^{(1)}, \cM^{(2)}$
and $\cM^{(3)}$, that is, $\cM_\rho := \cM^{(3)} \circ \cM^{(2)}_\rho 
\circ \cM^{(1)}$.

Since $A$ is positive definite, so is $A^{-1}$. Thus, both $\cM^{(1)}$ and
$\cM^{(3)}$ are monotone in the sense that they preserve the natural 
ordering of $\ell^\infty(\cT_h)$.  Moreover, since
\[
\frac{\p \tW_j}{\p W_{j-1}} = \rho \frac{\p \hF_\beta}{\p p_1}, \qquad
\frac{\p \tW_j}{\p W_j} = 1+ \rho \frac{\p \hF_\beta}{\p p_2}, \qquad
\frac{\p \tW_j}{\p W_{j+1}} = \rho \frac{\p \hF_\beta}{\p p_3}, 
\]
then the g-monotonicity of $\hF_\beta$ implies that
\[
\frac{\p \tW_j}{\p W_{j-1}}> 0,\qquad \frac{\p \tW_j}{\p W_{j+1}}> 0,
\qquad\mbox{and}\qquad \frac{\p \tW_j}{\p W_j}> 0
\]
provided that 
%\begin{equation}\label{e4.55}
%0< \rho < \left[2\alpha + \max_{p\in [p_{\mbox{\tiny min}}, 
%p_{\mbox{\tiny max}}]}\left(-\beta_2\frac{\p F}{\p p}\right) \right]^{-1}.
%\end{equation}
\begin{equation}\label{e4.55}
0< \rho < \left[2\alpha + \beta_2/\gamma \right]^{-1}.
\end{equation}
Thus, $\cM^{(2)}_\rho$ is monotone,  so is
$\cM_\rho := \cM^{(3)} \circ \cM^{(2)}_\rho \circ \cM^{(1)}$,
provided that $\rho$ satisfies \eqref{e4.55}.  The proof is complete. 
\end{proof}

\begin{theorem}\label{thm4.2}
Under the assumptions of Proposition \ref{prop4.1}, the finite difference scheme
\eqref{fdm} with $\hF=\hF_\beta$ is admissible and stable. 
\end{theorem}

\begin{proof}
By the definition of $\bG(\bU)$, we immediately have $\bG(\bU +\lambda)
=\bG(\bU)$ for any constant $\lambda$. Hence, $\cM_\rho(\bU +\lambda)
=\cM_\rho(\bU) +\lambda$, and we have $\cM_\rho$ commutes with the addition 
of constants.  Together with the monotonicity of $\cM_\rho$, it follows that
$\cM_\rho$ is nonexpansive in $\ell^\infty(\cT_h)$ 
(see \cite{Crandall_Tartar79}). Hence \eqref{A3} holds with 
$C=\max\{ |u_a|, |u_b|\}$, and we have the scheme is stable.

To prove admissibility of the scheme, let 
\begin{equation}\label{e4.56}
\delta_x^2\tV_j =\delta_x^2 V_j 
+ \rho \hF_\beta\bigl(\delta_x^2 V_{j-1},\delta_x^2 V_j,
\delta_x^2 V_{j+1}, x_j\bigr), \quad j=2,3,\cdots, J-1.
\end{equation}
Subtracting \eqref{e4.56} from \eqref{e4.50} and using the
mean value theorem we get
\begin{align}\label{e4.58}
\delta_x^2(\tU_j-\tV_j) &= \Bigl[1+ \rho \frac{\p \hF_\beta}{\p p_2} \Bigr]\,
\delta_x^2(U_j-V_j)+\rho\,\frac{\p \hF_\beta}{\p p_1}\,\delta_x^2(U_{j-1}-V_{j-1})
\\
&\hskip 0.4in
+\rho\,\frac{\p \hF_\beta}{\p p_3}\, \delta_x^2(U_{j+1}-V_{j+1}).\nonumber
\end{align}
Hence,
\begin{equation}\label{e4.59}
\|\tbU-\tbV\|_{\ell^\infty} 
	\leq \big( 1 + \rho \left[ (\beta_1 + \beta_3)\gamma - 1/\gamma \right] \big) 
		\|\bU-\bV\|_{\ell^\infty} 
	\leq \frac{1}{2} \|\bU-\bV\|_{\ell^\infty},
\end{equation}
which holds for
$(\beta_1 + \beta_3) < 1 / \gamma^2$ and 
$\rho \geq \frac{1}{2} \big[1/\gamma - (\beta_1 + \beta_3) \gamma \big]^{-1}$.
%\begin{equation}\label{e4.59}
%\|\tbU-\tbV\|_{\ell^\infty} \leq \frac12 \|\bU-\bV\|_{\ell^\infty}
%\end{equation}
%provided that
%\[
%1+ \rho \frac{\p \hF_\beta}{\p p_2}\leq \frac14, 
%\qquad \rho \frac{\p \hF_\beta}{\p p_1} \leq \frac18,\qquad
%\rho\,\frac{\p \hF_\beta}{\p p_3}\leq \frac18,  
%\]
%which hold for sufficiently small $\rho>0$ because $\hF_\beta$ 
%is g-monotone. 
Thus, \eqref{e4.59} implies that the mapping 
$\cM_\rho$ is contractive. By the fixed point theorem we conclude
that $\cM_\rho$ has a unique fixed point $U$, which in turn is the 
unique solution to the finite difference scheme \eqref{fdm}
with $\hF=\hF_\beta$. The proof is complete. 
\end{proof}

\begin{remark} \label{f2_stable}
We note that the choice $\beta_1 = \beta_3 = 0$ and $\beta_2 = 1$ 
trivially satisfies all of the restrictions in the proofs for any $\alpha > 0$.
We also note that the role of the numerical moment will be further 
explored numerically for degenerate elliptic test problems in section~\ref{sec-5}
\end{remark}
%\begin{remark}
%The constructive proof of Theorem \ref{thm4.2} provides
%a practical and convergent fixed point iterative algorithm for 
%solving nonlinear algebraic system \eqref{fdm}. 
%\end{remark}

%%%%%%%%%%%%%%%
%\section{Numerical experiments}\label{sec-5}

\section{Numerical Experiments}\label{sec-5}

In this section, we perform a series of numerical tests to demonstrate 
the accuracy and the order of convergence for the various proposed numerical schemes.  
As before, we assume a uniform mesh.  We use the Matlab built-in nonlinear 
solver {\em fsolve} for all tests, and, unless otherwise stated, we 
fix the initial guess $U^{(0)}$ as the linear interpolant of the boundary data.  
Also, all errors are measured in the $L^\infty$ norm.

For most tests, we record the results using $\widehat{F}_1$ and $\widehat{F}_4$.
Unless otherwise stated, the results for all of the proposed Lax-Friedrichs-like operators 
are analogous and the results for all of the proposed Godunov-like operators are analogous, 
even though the analysis that prompts Remark~\ref{f2_stable} suggests $\widehat{F}_2$ 
could be considered preferable to $\widehat{F}_1$ and $\widehat{F}_3$. 
For most of the examples we observe quadratic rates of convergence to the viscosity solution 
for the Lax-Friedrichs-like schemes.
For both classes of numerical operators we observe the lack of numerical artifacts that are 
known to plague the standard FD discretization for fully nonlinear problems.
However, for the Godunov-like schemes, this phenomena typically presents itself through the fact that 
the nonlinear solver {\em fsolve} fails to find a root.
Thus, while both classes of schemes support the selectivity of the discretizations, the 
resulting nonlinear algebraic system appears to be better suited for {\em fsolve} when 
using the Lax-Friedrichs-like operators.

We begin with a simple power nonlinearity that has a $C^\infty$ solution.

%%%%%%%%%%%%%%%%%%%%%%%%%%%%%%%%%%%%%%%%%%%%%%
%	Tests 1
%%%%%%%%%%%%%%%%%%%%%%%%%%%%%%%%%%%%%%%%%%%%%%

\medskip
{\bf Example 1:} Consider the problem
\begin{align*}
	- u_{x x}^3 + x^3 & =  0,	\qquad  -1 < x < 1, \\ 
u(-1)   = -1/6, \quad  u(1) & = 1/6,                   
\end{align*}
with the exact solution $u (x) = \frac{x^3}{6}$.

Using the linear interpolant of the boundary data as our initial guess and 
approximating $u$ with the various schemes above, we obtain the computed
results of Table \ref{table1} and Figure \ref{Fig1}.

\begin{table}[hbt]
\begin{center}
\begin{tabular}{| c | c | c | c | c |}
		\hline
	& \multicolumn{2}{|c|}{$\widehat{F}_1$ , $\alpha = 1.5$} &
		\multicolumn{2}{|c|}{$\widehat{F}_4$} \\ 
		\cline{2-5}
	$h$ & $L^\infty$ error & order & $L^\infty$ error & order \\ 
		\hline \cline{1-5}
	1.0000e-01 & 2.71e-02 & & 6.40e-02 & \\ 
		\hline
	5.0000e-02 & 5.10e-03 & 2.41 & 6.40e-02 & 0.00 \\ 
		\hline
	2.5000e-02 & 1.03e-03 & 2.31 & 6.40e-02 & 0.00 \\ 
		\hline
	1.2500e-02 & 2.33e-04 & 2.14 & 1.07e-03 & 5.90 \\ 
		\hline
	6.2500e-03 & 5.58e-05 & 2.06 & 2.12e-02 & -4.31 \\ 
		\hline
%	3.1250e-03 & 3.30e-02 & -9.21 & 6.55e-02 & -1.63 \\ 
%		\hline
\end{tabular}
\caption{Rates of convergence of Example 1.} \label{table1}
\end{center}
\end{table}

\begin{figure}[htb]
\centerline{
\includegraphics[scale=0.34]{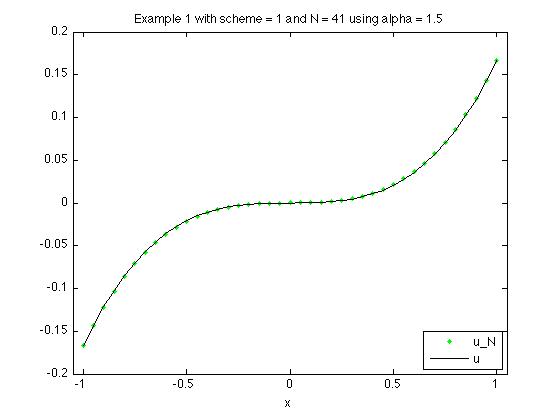}
%\hspace{5mm}
\includegraphics[scale=0.34]{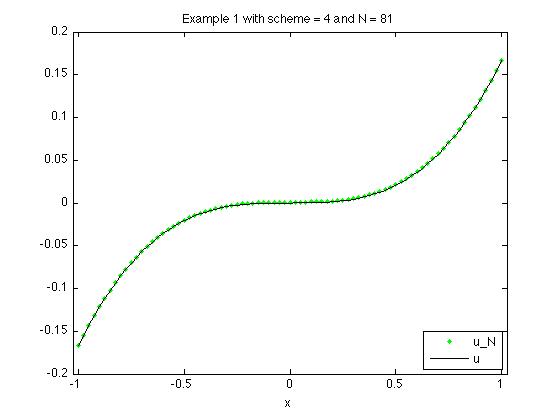}
}
\caption{Computed solutions of Example 1} \label{Fig1}
\end{figure}

The schemes $\widehat{F}_2$ and $\widehat{F}_3$ exhibit similar behavior 
as $\widehat{F}_1$, and $\widehat{F}_5$ exhibits similar behavior 
as $\widehat{F}_4$.  Thus, the Lax-Friedrichs-like schemes do exhibit 
a quadratic order of convergence as expected. On the other hand, 
the Godunov-like schemes converge inconsistently.  
This inconsistency is mostly due to {\em fsolve} failing to find a root.

If we fix our initial guess as the approximation computed by $\widehat{F}_1$ 
with $\alpha = 1.5$ and $h = 0.1$, we get the results of Table \ref{table2}.

\begin{table}[htb]
\begin{center}
\begin{tabular}{| c | c | c | c | c |}
		\hline
	& \multicolumn{2}{|c|}{$\widehat{F}_1$ , $\alpha = 1.5$} &
		\multicolumn{2}{|c|}{$\widehat{F}_4$} \\ 
		\cline{2-5}
	$h$ & $L^\infty$ error & order & $L^\infty$ error & order \\ 
		\hline \cline{1-5}
	1.0000e-01 & 2.71e-02 & & 8.24e-08 & \\ 
		\hline
	5.0000e-02 & 5.10e-03 & 2.41 & 1.58e-06 & -4.26 \\ 
		\hline
	2.5000e-02 & 1.03e-03 & 2.31 & 1.60e-05 & -3.34 \\ 
		\hline
	1.2500e-02 & 2.33e-04 & 2.14 & 9.06e-05 & -2.51 \\ 
		\hline
	6.2500e-03 & 5.58e-05 & 2.06 & 1.42e-02 & -7.29 \\ 
		\hline
%	3.1250e-03 & 3.74e-02 & -9.39 & 2.78e-02 & -0.97 \\ 
%		\hline
\end{tabular}
\end{center}
\caption{Rates of convergence of Example 1.} \label{table2}
\end{table}

Thus, Godunov-like schemes converge with high levels of accuracy when the 
nonlinear solver has a sufficiently good initial guess.  Since the Godunov-like 
schemes are very sensitive towards the initial guess for {\em fsolve}, it is hard 
to characterize a rate of convergence. We also observe in Table \ref{table1} that 
the error for $h = 0.1, 0.05, 0.025, 0.00625$ is consistent with the 
error of the initial guess for the Godunov-like schemes. In contrast, the 
Lax-Friedrichs-like schemes converge for a much wider range of initial guesses.

\medskip
The next example concerns the 1-D Monge-Ampere equation.

%%%%%%%%%%%%%%%%%%%%%%%%%%%%%%%%%%%%%%%%%%%%%%
%	Test 2
%%%%%%%%%%%%%%%%%%%%%%%%%%%%%%%%%%%%%%%%%%%%%%

\medskip
{\bf Example 2:} Consider the problem
\begin{align*}
	- u_{x x}^2 + 1 & =  0,	\qquad 0 < x < 1, \\ 
	u(0)   = 0,\quad  u(1) & = 1/2 .                      
\end{align*}
This problem has exactly two solutions
\begin{align*}
	u^+ (x) = \frac{1}{2} x^2, \qquad u^- (x) = - \frac{1}{2} x^2 + x ,
\end{align*}
where $u^+$ is convex and $u^-$ is concave.  However, $u^+$ is the unique 
viscosity solution that preserves the ellipticity of the operator.

Using $U^{(0)}$ as the linear interpolant of the boundary data, 
the computed results with both types of schemes are given in Table \ref{table3}.

\begin{table}[htb]
\begin{center}
\begin{tabular}{| c | c | c | c | c | c | c |}
		\hline
	& \multicolumn{2}{|c|}{$\widehat{F}_{1}$ , $\alpha = 1$} &
	\multicolumn{2}{|c|}{$\widehat{F}_{1}$ , $\alpha = -1$} &
	\multicolumn{2}{|c|}{$\widehat{F}_{4}$} \\ 
		\cline{2-7}
$h$ & $L^\infty$ error & order & $L^\infty$ error & order & $L^\infty$ error & order \\ 
		\hline \cline{1-7}
	1.000e-01 & 2.54e-03 & & 2.54e-03 & & 1.17e-01 & \\ 
		\hline
	5.000e-02 & 6.36e-04 & 2.00 & 6.36e-04 & 2.00 & 1.21e-01 & -0.05 \\ 
		\hline
	2.500e-02 & 1.59e-04 & 2.00 & 1.59e-04 & 2.00 & 1.24e-01 & -0.04 \\ 
		\hline
%	1.250e-02 & 1.17e-01 & -9.52 & 3.97e-05 & 2.00 & 1.25e-01 & -0.01 \\ 
%		\hline
\end{tabular} 
\end{center}
\caption{Rates of convergence of Example 2.} \label{table3}
\end{table}

\begin{figure}[htb]
\centerline{
\includegraphics[scale=0.35]{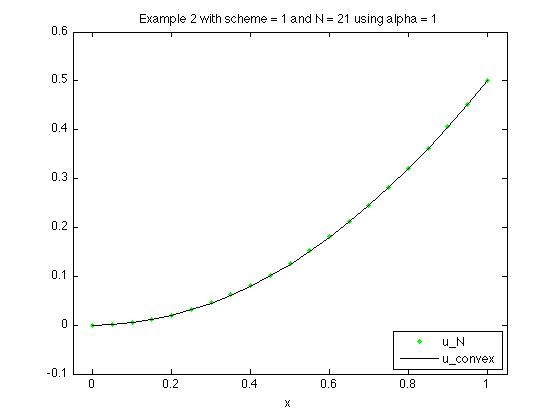}
%\hspace{5mm}
\includegraphics[scale=0.35]{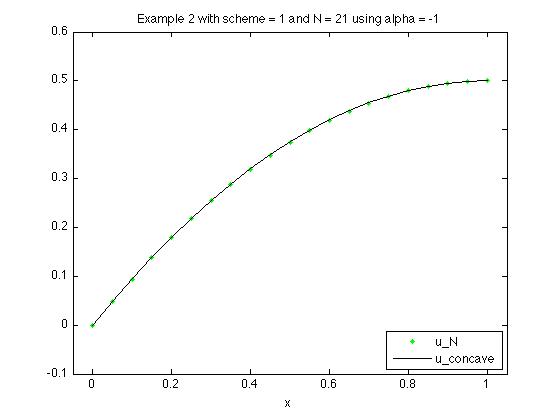}
} 
\begin{center}
\includegraphics[scale=0.38]{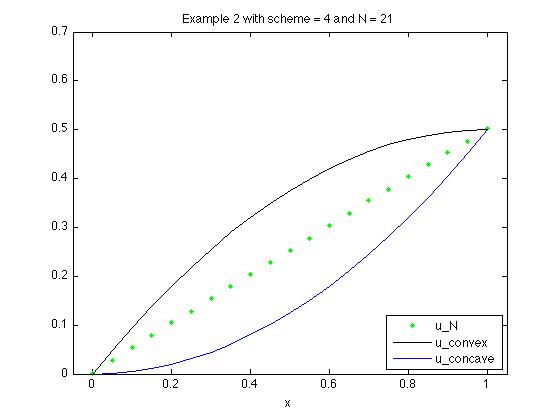}
\end{center}
\caption{Computed solutions of Example 2 with different parameter $\alpha$.} 
\label{fig2}
\end{figure}

We note that the Lax-Friedrichs-like schemes converge to the unique 
ellipticity preserving solution (i.e., convex solution)
for $\alpha > 0$ sufficiently large. However, if $\alpha < 0$ 
with $\left| \alpha \right|$ sufficiently large, 
the Lax-Friedrichs-like schemes converge to $u^-$.  The convergence 
to $u^-$ for $\alpha < 0$ is expected since $u^-$ is the unique solution 
that preserves the ellipticity of the PDE $u_{x x}^2 - 1 = 0$.  Forming the 
corresponding Lax-Friedrichs-like scheme and multiplying by $-1$ is equivalent 
to letting $\alpha < 0$ in the above formulation. 
%Lastly, the Lax-Friedrichs-like schemes do not converge when 
%$\alpha<0$, which is expected.

We also test the benefit of using a Lax-Friedrichs-like scheme as 
opposed to the standard $3$-point finite difference method. We approximate $u$ 
using $\widehat{F}_2$ for varying values of $\alpha$, using the linear 
interpolant of the boundary data as our initial guess. The 
computed results are given in Table \ref{table4}.

\begin{table}[htb]
\begin{center}
\begin{tabular}{| c | c | c | c | c | c | c |}
		\hline
	& \multicolumn{2}{|c|}{$\widehat{F}_{2}$ , $\alpha = 6$} &
	\multicolumn{2}{|c|}{$\widehat{F}_{2}$ , $\alpha = 0.05$} &
		\multicolumn{2}{|c|}{$\widehat{F}_{2}$ , $\alpha = 0$} \\ 
		\cline{2-7}
$h$ & $L^\infty$ error & order & $L^\infty$ error & order & $L^\infty$ error & order \\ 
		\hline \cline{1-7}
	1.000e-01 & 3.07e-02 & & 1.18e-01 & & 9.00e-02 & \\ 
		\hline
	5.000e-02 & 8.51e-03 & 1.85 & 3.31e-02 & 1.83 & 1.15e-01 & -0.35 \\ 
		\hline
	2.500e-02 & 2.14e-03 & 1.99 & 3.03e-02 & 0.13 & 1.15e-01 & -0.00 \\ 
		\hline
\end{tabular}
\end{center}
\caption{Performance of a Lax-Friedrichs-like scheme with various $\alpha$}
\label{table4}
\end{table}

\begin{figure}{hb}
\centerline{
\includegraphics[scale=0.34]{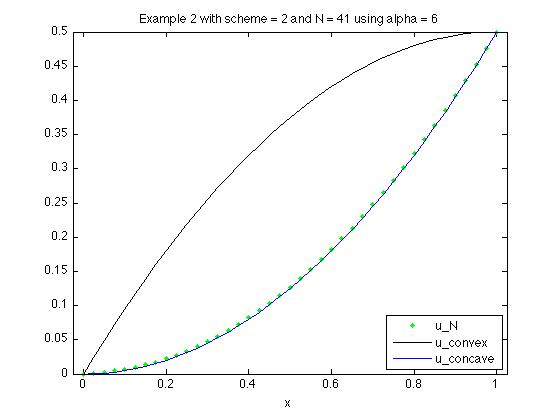}
%\hspace{5mm}
\includegraphics[scale=0.34]{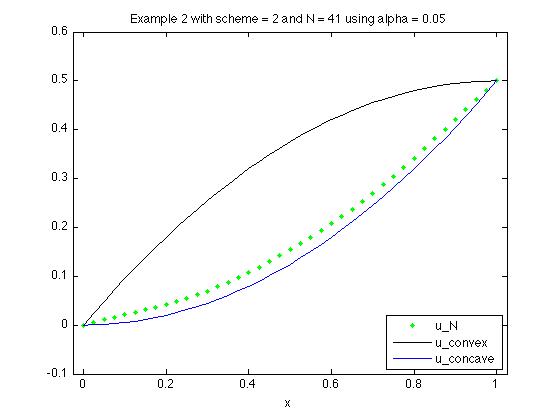}
}
\begin{center}
\includegraphics[scale=0.36]{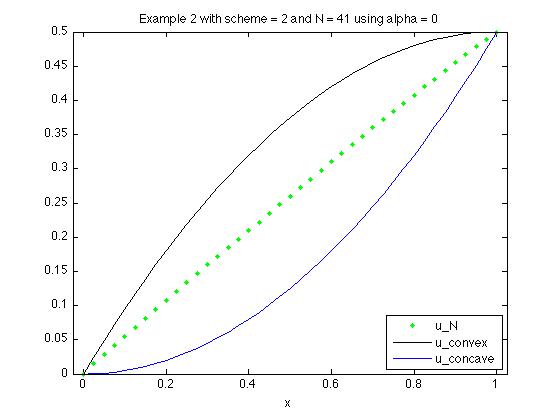}
\end{center}
\caption{Computed solutions by a Lax-Friedrichs-like scheme with various $\alpha$}
\label{Fig3}
\end{figure}
 
We remark that letting $\alpha = 0$ corresponds to the standard $3$-point
finite difference method, which does not converge in the above example.  
Instead, it behaves similarly to the Godunov-like schemes in that the nonlinear 
solver cannot determine a good direction to move from the initial guess.
Thus, the Lax-Friedrichs-like schemes have a mechanism for giving the 
nonlinear solver a good direction towards finding a root.  
When $\alpha$ is sufficiently large, the schemes converge.
When $\alpha$ is not sufficiently large, while the schemes may not converge, 
they have a tendency to move towards the correct solution.  
Furthermore, we can see that the Lax-Friedrichs-like schemes converge 
quadratically for $\alpha$ bigger than the theoretical lower bound with 
only a small cost in the level of accuracy. Thus, when dealing with a problem 
that has an unknown optimal bound for $\alpha$, large $\alpha$ values can be used.  
A shooting method for decreasing $\alpha$ allows the scheme to gain accuracy 
while maintaining the benefits of the Lax-Friedrichs-like schemes.

If we first use $\widehat{F}_1$ with $\alpha = 1$ to approximate $u$ on a 
coarse mesh with $h = 0.1$, and then we interpolate the result to get an initial 
guess for the two proposed schemes and the $3$-point finite difference method, 
we get the results of Table \ref{table5}.
Thus, we see that the Godunov-like schemes and the standard
finite difference formulation now converge to $u^+$ with high levels of
accuracy given a sufficiently good initial guess.  In
fact, they both converge to the same limit.

\begin{table}[bt]
\begin{center}
\begin{tabular}{| c | c | c | c | c | c | c |}
		\hline
	& \multicolumn{2}{|c|}{$\widehat{F}_{1}$ , $\alpha = 1$} &
	\multicolumn{2}{|c|}{$\widehat{F}_{4}$} &
		\multicolumn{2}{|c|}{$\widehat{F}_{2}$ , $\alpha = 0$} \\ 
		\cline{2-7}
$h$ & $L^\infty$ error & order & $L^\infty$ error & order & $L^\infty$ error & order \\ 
		\hline \cline{1-7}
	1.000e-01 & 2.54e-03 & & 9.96e-15 & & 9.96e-15 & \\ 
		\hline
	5.000e-02 & 6.36e-04 & 2.00 & 4.54e-13 & -5.51 & 4.54e-13 & -5.51 \\ 
		\hline
	2.500e-02 & 1.59e-04 & 2.00 & 1.46e-10 & -8.33 & 1.46e-10 & -8.33 \\ 
		\hline
	1.250e-02 & 3.97e-05 & 2.00 & 9.85e-10 & -2.75 & 9.85e-10 & -2.75 \\ 
		\hline
\end{tabular}
\end{center}
\caption{Performance of the standard $3$-point scheme}\label{table5}
\end{table}

\begin{figure}[htb]
\centerline{
\includegraphics[scale=0.34]{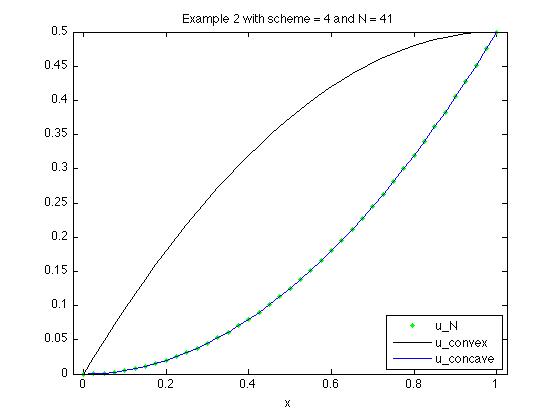}
%\hspace{5mm}
\includegraphics[scale=0.34]{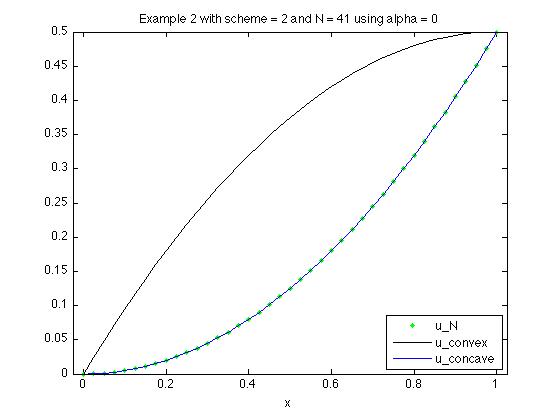}
}
\caption{Computed solutions by a Godunov-like scheme and the standard $3$-point scheme}
\label{Fig4}
\end{figure}

To the contrary, if we use $\widehat{F}_1$ with $\alpha = -1$ to approximate
$u$ on a coarse mesh with $h = 0.1$ and then interpolate the result as an 
initial guess, we obtain the results of Table \ref{table6}.
Clearly, none of the schemes converge to $u^+$. Moreover, the Lax-Friedrichs-like
schemes and the Godunov-like schemes do not converge to $u^-$ even if $U^{0}$
is close to $u^-$.  Instead, {\em fsolve} finds no solution 
when using the two proposed schemes. Thus, the Lax-Friedrichs-like 
schemes and the Godunov-like schemes appear to only consider $u^+$ to be the 
solution of the PDE. Since $u^+$ is the unique viscosity solution
of the PDE, lack of convergence to $u^-$ for
the Lax-Friedrichs-like schemes for $\alpha > 0$ sufficiently
large and for the Godunov-like schemes is consistent with theory.

In contrast, the standard $3$-point finite difference method does converge to $u^-$.
When given a sufficiently good guess, the $3$-point finite difference method
will converge to any one of the two solutions. Furthermore, the discretization can
create artificial solutions that will attract the standard $3$-point finite difference
method. On the other hand, the monotonicity of our proposed schemes
prevent the discretizations from having multiple solutions.

\begin{table}[htb]
\begin{center}
\begin{tabular}{| c | c | c | c | c | c | c |}
		\hline
	& \multicolumn{2}{|c|}{$\widehat{F}_{1}$ , $\alpha = -1$} &
	\multicolumn{2}{|c|}{$\widehat{F}_{4}$} &
		\multicolumn{2}{|c|}{$\widehat{F}_{2}$ , $\alpha = 0$} \\ 
		\cline{2-7}
$h$ & $L^\infty$ error & order & $L^\infty$ error & order & $L^\infty$ error & order \\ 
		\hline \cline{1-7}
	1.000e-01 & 2.68e-02 & & 2.56e-03 & & 2.24e-14 & \\ 
		\hline
	5.000e-02 & 5.61e-03 & 2.25 & 2.54e-03 & 0.01 & 8.82e-13 & -5.30 \\ 
		\hline
	2.500e-02 & 1.26e-02 & -1.16 & 2.54e-03 & 0.00 & 8.83e-12 & -3.32 \\ 
		\hline
	1.250e-02 & 1.41e-02 & -0.17 & 2.54e-03 & -0.00 & 1.63e-09 & -7.53 \\ 
		\hline
\end{tabular}
\end{center}
\caption{Performance of a Lax-Friedrichs-like scheme with $\alpha = -1$. }
\label{table6}
\end{table}

\begin{figure}[htb]
\centerline{
\includegraphics[scale=0.17]{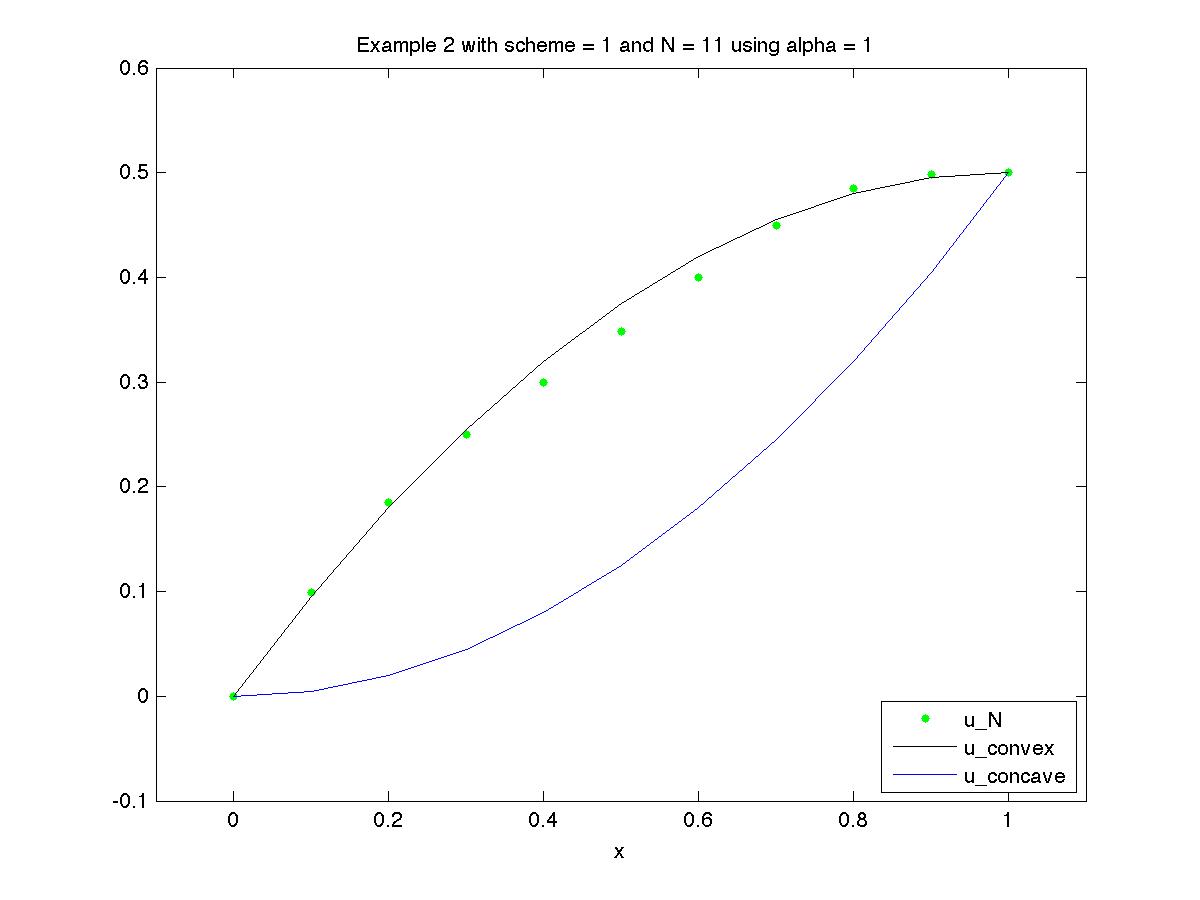}
%\hspace{5mm}
\includegraphics[scale=0.34]{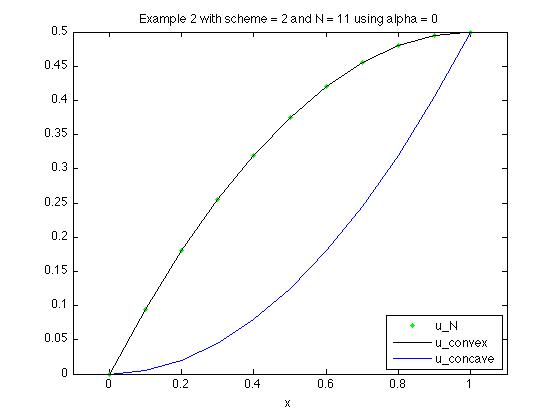}
}
\caption{Computed solutions of a Lax-Friedrichs-like scheme with $\alpha = -1$. }
\label{Fig5}
\end{figure}

The next two examples deals with Bellman type equations.

%%%%%%%%%%%%%%%%%%%%%%%%%%%%%%%%%%%%%%%%%%%%%%
%	Test 3
%%%%%%%%%%%%%%%%%%%%%%%%%%%%%%%%%%%%%%%%%%%%%%

{\bf Example 3:} Consider the problem
\begin{align*}
\min_{\theta(x) \in \left\{ 1, 2 \right\}} \bigl\{- A_\theta u_{x x} 
- S \left( x \right)\bigr\} & =  0, \qquad -1 < x < 1, \\ 
	u(-1)   = -1, \quad u(1) & = 1.                      
\end{align*}
for
\begin{align*}
	A_1  = 1,\quad
	A_2  = 2,\quad
	S (x)  = 
		\begin{cases}
			12 x^2, & \text{if } x < 0, \\
			-24 x^2, & \text{if } x \geq 0.
		\end{cases} 
\end{align*}
This problem has the exact solution $u (x) = x \left| x \right|^3$.
We also note that this problem has a finite dimensional control parameter set.

Using the linear interpolant as the initial guess, we obtain the results
of Table \ref{table7}. We observe that the Godunov-like scheme converges
and both schemes exhibit quadratic convergence for this example.  
%In general, it is expected that the Godunov-like schemes have a linear 
%rate of convergence.

\begin{table}[htb]
\begin{center}
\begin{tabular}{| c | c | c | c | c |}
		\hline
	& \multicolumn{2}{|c|}{$\widehat{F}_{1}$ , $\alpha = 1$} &
	\multicolumn{2}{|c|}{$\widehat{F}_{4}$} \\ 
		\cline{2-5}
	$h$ & $L^\infty$ error & order & $L^\infty$ error & order \\ 
		\hline \cline{1-5}
	1.000e-01 & 1.29e-01 & & 9.60e-03 & \\ 
		\hline
	5.000e-02 & 4.67e-02 & 1.46 & 2.50e-03 & 1.94 \\ 
		\hline
	2.500e-02 & 1.46e-02 & 1.68 & 6.25e-04 & 2.00 \\ 
		\hline
	1.250e-02 & 4.18e-03 & 1.80 & 4.70e-01 & -9.55 \\ 
		\hline
	6.250e-03 & 1.13e-03 & 1.89 & 4.72e-01 & -0.01 \\ 
		\hline
	3.125e-03 & 2.95e-04 & 1.93 & 4.72e-01 & -0.00 \\ 
		\hline
\end{tabular}
\end{center}
\caption{Rates of convergence of Example 3.}\label{table7}
\end{table}

\begin{figure}[hbt]
\centerline{
\includegraphics[scale=0.34]{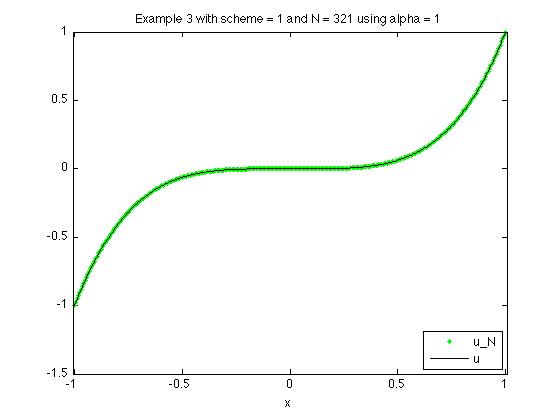}
%\hspace{5mm}
\includegraphics[scale=0.34]{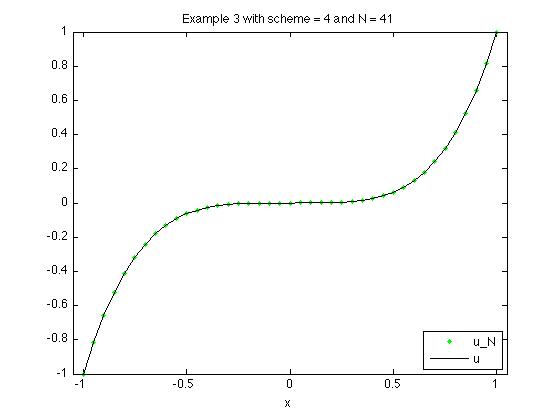}
} 
\caption{Computed solutions of Example 3.}\label{Fig6}
\end{figure}

\medskip
Now we consider a Bellman problem with infinite dimensional control parameter set.

%%%%%%%%%%%%%%%%%%%%%%%%%%%%%%%%%%%%%%%%%%%%%%
%	Test 4
%%%%%%%%%%%%%%%%%%%%%%%%%%%%%%%%%%%%%%%%%%%%%%

\medskip
{\bf Example 4:}
Let $\theta : \mathbb{R} \to \mathbb{R}$ such that $\theta \in L^\infty ([ 2, 4])$, 
and consider the problem
\begin{align*}
\inf_{ -1 \leq \theta(x) \leq 1} \bigl\{- \theta \, u_{x x} + \theta^2 \, u + x^{-2} 
\bigr\} & =  0,	\qquad 2 < x < 4, \\ 
	u(2)   = 4,\quad  u(4) & = 16 .                      
\end{align*}
This problem has the exact solution $u (x) = x^2$ with the corresponding 
control $\theta (x) = x^{-2}$.

Let the initial guess be given by the linear interpolant of the boundary data.  
Then, we obtain the results of Table \ref{table8}.

\begin{table}[htb]
\begin{center}
\begin{tabular}{| c | c | c | c | c |}
		\hline
	& \multicolumn{2}{|c|}{$\widehat{F}_{1}$ , $\alpha = 0.5$} &
	\multicolumn{2}{|c|}{$\widehat{F}_{4}$} \\ 
		\cline{2-5}
	$h$ & $L^\infty$ error & order & $L^\infty$ error & order \\ 
		\hline \cline{1-5}
	1.000e-01 & 3.07e-01 & & 5.59e-01 & \\ 
		\hline
	5.000e-02 & 9.88e-02 & 1.64 & 4.96e-01 & 0.17 \\ 
		\hline
	2.500e-02 & 3.09e-02 & 1.68 & 5.10e+00 & -3.36 \\ 
		\hline
%	1.250e-02 & 7.86e-01 & -4.67 & 2.03e+00 & 1.33 \\ 
%		\hline
\end{tabular}
\end{center}
\caption{Rates of convergence of Example 4.}\label{table8}
\end{table}
Both schemes have a hard time finding a root for $h$ small, although the 
Lax-Friedrichs-like schemes do converge towards $u$ for larger values of $h$.

Now we choose the initial guess 
\[
	U^{(0)} = \frac{3}{14} x^3 + \frac{16}{7} , 
\]
a simple cubic function that satisfies the boundary conditions.
Then, $\|U^{(0)}$ $-u \|_{L^\infty ( [2, 4] )} \approx 0.94$, 
and we get the results of Table \ref{table9}.

\begin{table}[htb]
\begin{center}
\begin{tabular}{| c | c | c | c | c |}
		\hline
	& \multicolumn{2}{|c|}{$\widehat{F}_{1}$ , $\alpha = 0.5$} &
	\multicolumn{2}{|c|}{$\widehat{F}_{4}$} \\ 
		\cline{2-5}
	$h$ & $L^\infty$ error & order & $L^\infty$ error & order \\ 
		\hline \cline{1-5}
	1.000e-01 & 3.07e-01 & & 6.74e-10 & \\ 
		\hline
	5.000e-02 & 9.88e-02 & 1.64 & 7.04e-08 & -6.71 \\ 
		\hline
	2.500e-02 & 3.09e-02 & 1.68 & 3.41e-09 & 4.37 \\ 
		\hline
	1.250e-02 & 9.02e-03 & 1.78 & 8.09e-08 & -4.57 \\ 
		\hline
	6.250e-03 & 2.47e-03 & 1.87 & 9.44e-01 & -23.48 \\ 
		\hline
\end{tabular}
\end{center}
\caption{Rates of convergence of Example 4.}\label{table9}
\end{table}

\begin{figure}[htb]
\centerline{
\includegraphics[scale=0.34]{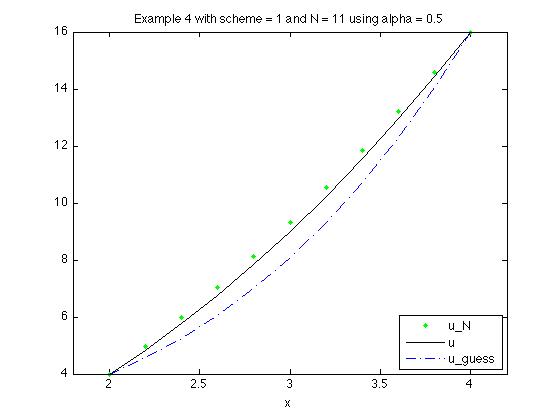}
%\hspace{5mm}
\includegraphics[scale=0.34]{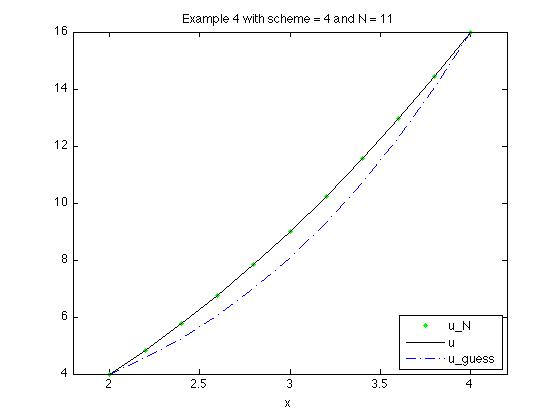}
} 
\caption{Computed solutions of Example 4.}\label{Fig7}
\end{figure}
Thus, the Lax-Friedrichs-like schemes again converge with a rate of almost $2$.
Also, the Godunov-like schemes converge with high levels of accuracy for 
$h \geq 0.0125$, but for smaller $h$, {\em fsolve} fails to find a root .

We remark that this problem can also be approximated by using a splitting algorithm.  
The operator can be split into an optimization problem for $\theta$ and a 
linear PDE problem for $u$, and then a natural scheme is to successively 
approximate $\theta$ and $u$ starting with an initial guess for $\theta$.  
For the above approximations, the nonlinearity due to the infimum was preserved 
inside the definition of the operator.  

\medskip
The final example considers a problem whose solution is not classical.
 
%%%%%%%%%%%%%%%%%%%%%%%%%%%%%%%%%%%%%%%%%%%%%%
%	Test 5
%%%%%%%%%%%%%%%%%%%%%%%%%%%%%%%%%%%%%%%%%%%%%%

{\bf Example 5:} Consider the problem
\begin{align*}
	- u_{x x}^{3} + 8 \text{ sign}(x) & =  0,	\qquad -1 < x < 1, \\ 
          u(-1)  = -1, \quad  u(1) & = 1 ,                 
\end{align*}
with the exact solution $u (x) = x |x| \in C^1 ( [-1,1] )$. 

Using the linear interpolant of the boundary data as the initial guess, 
we obtain the results of Table \ref{table10}.
 
\begin{table}[htb]
\begin{center}
\begin{tabular}{| c | c | c | c | c |}
		\hline
	& \multicolumn{2}{|c|}{$\widehat{F}_{1}$ , $\alpha = 1.5$} &
	\multicolumn{2}{|c|}{$\widehat{F}_{4}$} \\ 
		\cline{2-5}
	$h$ & $L^\infty$ error & order & $L^\infty$ error & order \\ 
		\hline \cline{1-5}
	1.000e-01 & 1.59e-02 & & 2.40e-01 & \\ 
		\hline
	5.000e-02 & 3.76e-03 & 2.08 & 2.50e-01 & -0.06 \\ 
		\hline
	2.500e-02 & 9.40e-04 & 2.00 & 2.50e-01 & 0.00 \\ 
		\hline
	1.250e-02 & 2.35e-04 & 2.00 & 6.69e-06 & 15.19 \\ 
		\hline
	6.250e-03 & 5.88e-05 & 2.00 & 2.05e-01 & -14.90 \\ 
		\hline
\end{tabular}
\end{center}
\caption{Rates of convergence of Example 5.}\label{table10}
\end{table}

\begin{figure}[htb]
\centerline{
\includegraphics[scale=0.34]{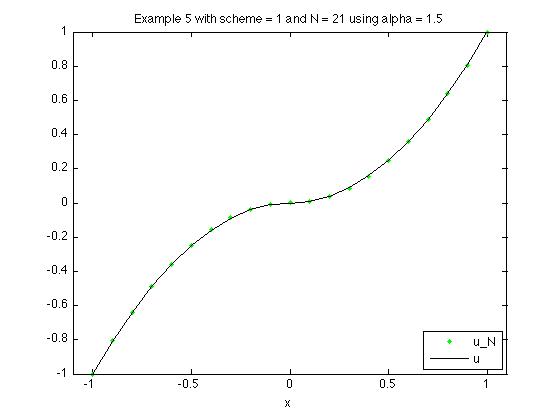}
%\hspace{5mm}
\includegraphics[scale=0.34]{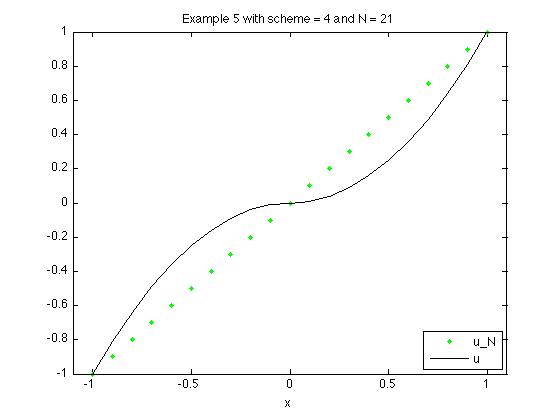}
}
\centerline{
\includegraphics[scale=0.34]{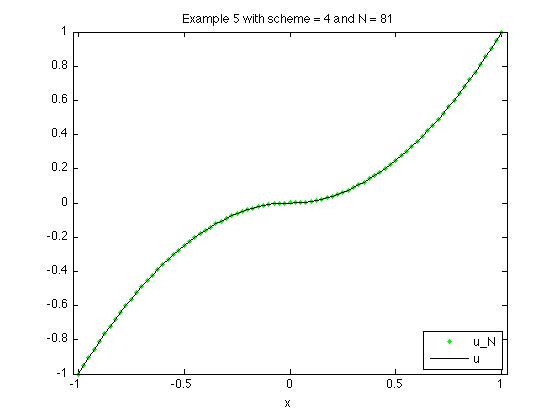}
%\hspace{5mm}
\includegraphics[scale=0.34]{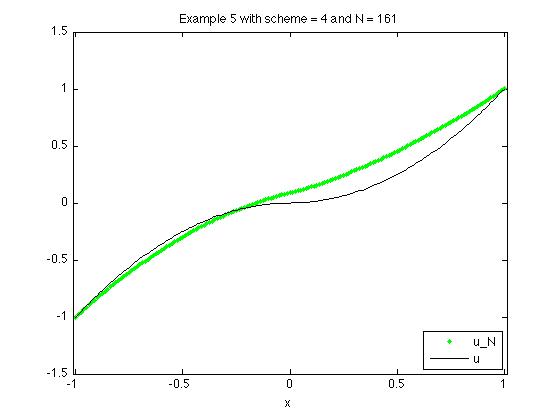}
}
\caption{Computed solutions of Example 4.}\label{Fig8}
\end{figure}
We clearly see the quadratic rate of convergence for the Lax-Friedrichs-like 
schemes.  The Godunov-like schemes only converge for $h = 0.0125$.  For 
larger $h$, the scheme returns the initial guess after failing to find a root.  
For the test with smaller $h$, the scheme returns a slightly improved 
approximation after reaching the maximum number of iterations.

If we fix our initial guess as the approximation formed by $\widehat{F}_1$ 
with $\alpha = 1.5$ and $h = 0.1$, we then get the results of Table \ref{table11}.

\begin{table}[htb]
\begin{center}
\begin{tabular}{| c | c | c | c | c |}
		\hline
	& \multicolumn{2}{|c|}{$\widehat{F}_{1}$ , $\alpha = 1.5$} &
	\multicolumn{2}{|c|}{$\widehat{F}_{4}$} \\ 
		\cline{2-5}
	$h$ & $L^\infty$ error & order & $L^\infty$ error & order \\ 
		\hline \cline{1-5}
	1.000e-01 & 1.59e-02 & & 1.84e-08 & \\ 
		\hline
	5.000e-02 & 3.76e-03 & 2.08 & 4.05e-06 & -7.78 \\ 
		\hline
	2.500e-02 & 9.40e-04 & 2.00 & 8.85e-06 & -1.13 \\ 
		\hline
	1.250e-02 & 2.35e-04 & 2.00 & 6.50e-06 & 0.45 \\ 
		\hline
	6.250e-03 & 5.88e-05 & 2.00 & 7.78e-06 & -0.26 \\ 
		\hline
\end{tabular}
\end{center}
\caption{Rates of convergence of Example 5.}\label{table11}
\end{table}
As observed in the previous examples, we see that the Godunov-like schemes 
converge quickly with high levels of accuracy, thus making it difficult to 
characterize a general rate of convergence.  

%%%%%%%%%%%%%%%%%%%%%%%%%%%%%%%%%%%%%%%%%%%%%%
%Test 6    DOES NOT WORK - |x| not viscosity solution for MA type problem
%%%%%%%%%%%%%%%%%%%%%%%%%%%%%%%%%%%%%%%%%%%%%%
%
%\noindent {\textbf{EXAMPLE 6 } :
%
%\vspace{5mm}
%
%For this example, we want to approximate a PDE with a viscosity solution that has a corner
%to demonstrate the Godunov-like scheme's ability to preserve corners. 
%Consider the problem
%\begin{align*}
%	- u_{x x}^2 - \big(u_x - \text{sign} (x) \big)^2 - \left( u - | x | \right)^2 & =  0,	\quad -1 < x < 1 \\ 
%	u(-1)  & = 1,                  \\
%	u(1) & =  1.                      
%\end{align*}
%This problem has the exact viscosity solution
%\[
%	u (x) = \left| x \right| .
%\]

%%%%%%%%%%%%%%%
\section{Conclusion}\label{sec-6}
We have presented a new framework for constructing and analyzing 
consistent, g-monotone, and stable finite difference methods.  The 
newly proposed consistency and g-monotonicity criterion are not only 
simple to understand, but they are also easy to verify in practice. The key 
concept of the framework is the ``numerical operator", which plays
the same role as the ``numerical Hamiltonian" does in the successful
monotone finite difference framework for first order fully
nonlinear Hamilton-Jacobi equations. To construct practically 
useful finite difference methods which can be easily implemented
on computers, we have also presented a guideline for designing finite 
difference methods which fulfill the structure criterion of the 
proposed finite difference framework. The key concept in this 
regard is the ``numerical moment", which plays the same role 
as the ``numerical viscosity" does in the successful finite difference 
framework for first order fully nonlinear Hamilton-Jacobi equations.
Moreover, we gave some numerical evidences and argued that 
``numerical moments" provide an indispensable mechanism and
ability for a finite difference scheme to be able  
to converge to the viscosity solution of the underlying second
order fully nonlinear PDE problem. To a certain degree, 
the work of this paper bridges the 
gap between the state-of-the-art of finite difference 
methods for second order fully nonlinear PDEs and that 
for first order fully nonlinear Hamilton-Jacobi equations.
Although the results of this paper are confined to the
one spatial dimension case, they are also expected to hold 
in high spatial dimensions; that result will be presented 
in a forthcoming companion paper \cite{Feng_Kao_Lewis11}.

%%%%%
\bibliographystyle{elsarticle-num}

\end{document}